\documentclass[12pt]{article}

\usepackage{amsmath}
\usepackage{mathtools}
\usepackage{amssymb}
\usepackage{enumerate}
\usepackage{here}
\usepackage{fullpage}
\usepackage[usenames,dvipsnames,svgnames,table]{xcolor}
\usepackage{mathrsfs}
\usepackage{multirow}

\usepackage{cite}
\def\citep{\cite}

\usepackage{dsfont}
\def\d{\textnormal{d}}
\def\Tmin{T_{\rm min}}
\def\Tmax{T_{\rm max}}

\providecommand{\blue}[1]{\color{black}{#1}\color{black}\hspace{0pt}}

\definecolor{lblue}{RGB}{0,110,152}

\definecolor{dred}{RGB}{171,67,53}

\newtheorem{theorem}{Theorem}

\newtheorem{proposition}[theorem]{Proposition}

\newtheorem{define}[theorem]{Definition}
\newtheorem{example}[theorem]{Example}

\newtheorem{remark}[theorem]{Remark}

\newtheorem{problem}[theorem]{Problem}

\newcommand{\mendth}{\hfill \ensuremath{\vartriangle}}

\DeclareMathOperator*{\col}{col}

\DeclareMathOperator*{\diag}{diag}

\DeclareMathOperator{\eps}{\varepsilon}

\newenvironment{proof}{{\it Proof :~}}{\hfill$\diamondsuit$\\}

\begin{document}

\title{$L_1/\ell_1$-to-$L_1/\ell_1$ analysis of linear positive impulsive systems with application to the $L_1/\ell_1$-to-$L_1/\ell_1$ interval observation of linear impulsive and switched systems}

\author{Corentin Briat\thanks{corentin@briat.info; http://www.briat.info}}


\maketitle

\begin{abstract}
Sufficient conditions characterizing the asymptotic stability and the hybrid $L_1/\ell_1$-gain of linear positive impulsive systems under minimum and range dwell-time constraints are obtained. These conditions are stated as infinite-dimensional linear programming problems that can be solved using sum of squares programming, a relaxation that is known to be asymptotically exact in the present case. These conditions are then adapted to formulate constructive and convex sufficient conditions for the existence of $L_1/\ell_1$-to-$L_1/\ell_1$ interval observers for linear impulsive and switched systems. Suitable observer gains can be extracted from the (suboptimal) solution of the infinite-dimensional optimization problem where the $L_1/\ell_1$-gain of the system mapping the disturbances to the weighted observation errors is minimized. Some examples on impulsive and switched systems are given for illustration.

\end{abstract}

\section{Introduction}


Interval observers have been first introduced in \cite{Gouze:00} in order to account for the presence of disturbances acting on the observed system. The key idea is to build two, possibly coupled, observers, whose goal will be the real-time estimation of a lower and an upper bound on the state of the system. The problem of designing such observers has been considered for various classes of systems, including systems with inputs \cite{Mazenc:11,Briat:15g}, linear systems \cite{Mazenc:12,Combastel:13,Cacace:15}, time-varying systems \cite{Thabet:14}, delay systems \cite{Efimov:13c,Briat:15g}, LPV systems \cite{Efimov:13b,Chebotarev:15}, discrete-time systems \cite{Mazenc:13,Briat:15g} and systems with samplings \cite{Mazenc:14b,Efimov:16}. In spite of the importance of hybrid systems \cite{Sivashankar:94,Naghshtabrizi:08,Briat:11l,Goebel:12,Briat:15i,Geromel:15,Briat:16c} as they are able to model of wide variety of real world processes, the interval observation problem has been relatively few studied for such systems, yet some results on impulsive systems \cite{Degue:16nolcos,Briat:17ifacObs} and switched systems \cite{Rabehi:17,Ethabet:17} exist.

A common approach to interval observation is to enforce that the dynamics of the observation errors be described by positive dynamics, thereby ensuring that the errors do not change sign over time, resulting then in the correct estimation of both an upper and a lower bound on the state of the system. In this regard, obtaining conditions for the design of interval observers for impulsive and switched systems requires stability results for linear impulsive and switched systems. Such systems have been notably considered in \cite{Briat:16c} where preliminary stability and stabilization conditions under various dwell-time constraints have been obtained and formulated in terms of, finite- or infinite-dimensional, linear programs which can then be solved using recent optimization techniques.

The main objective of the paper is twofold. The first one is the derivation of sufficient conditions characterizing the asymptotic stability and the hybrid $L_1/\ell_1$-gain of linear positive impulsive systems under minimum and range dwell-time constraints. The reason for considering the $L_1/\ell_1$-gain, which consists of a combination of the $L_1$-gain and the $\ell_1$-gain, stems from the fact that those gains arise as natural performance measures for linear positive systems when linear copositive Lyapunov functions are considered; see e.g. \cite{Briat:11h}. The overall method relies \blue{on an extension of the} linear programming stability conditions recently obtained in \cite{Briat:16c} for linear positive impulsive/switched systems referred to as clock-dependent conditions because of their dependence of the (relative) time elapsed since the last discrete-event (i.e. the last impulse time or the last switching time). These conditions are the linear programming analogues of the semidefinite programming conditions obtained in \cite{Briat:13d,Briat:14f,Briat:15f,Briat:15i} in the context of more general linear impulsive, sampled-data and switched systems. The newly obtained conditions have the benefits of being linear in the system matrices and being readily available for design purposes, which leads to the second objective of the paper: the derivation of convex conditions for the design of interval observers for linear impulsive and switched systems. Three results are obtained in this respect: the first one pertains on the interval observation of linear impulsive systems under a range dwell-time constraint whereas the second one considers a minimum dwell-time constraint. The third and last result is about the interval observation of linear switched systems under a minimum dwell-time constraint. All the obtained conditions can be checked using discretization methods \cite{Allerhand:11}, using linear programming via the use of Handelman's theorem \cite{Handelman:88,Briat:11h} or using semidefinite programming via the use of Putinar's Positivstellensatz \cite{Putinar:93} combined with computational sum of square methods \cite{Parrilo:00,sostools3}. The approach is finally illustrated through numerical examples.

\noindent\textbf{Notations.} The set of integers greater or equal to $n\in\mathbb{Z}$ is denoted by $\mathbb{Z}_{\ge n}$. The cones of positive and nonnegative vectors of dimension $n$ are denoted by $\mathbb{R}_{>0}^n$ and $\mathbb{R}_{\ge0}^n$, respectively.  For any matrix $M$, the inequalities $M\ge0$ and $M>0$ are always interpreted componentwise. The set of diagonal matrices of dimension $n$ is denoted by $\mathbb{D}^n$ and the subset of those being positive definite is denoted by $\mathbb{D}_{\succ0}^n$. The $n$-dimensional vector of ones is denoted by $\mathds{1}_n$. The dimension will be often omitted as it is obvious from the context. For some elements, $\{x_1,\ldots,x_n\}$, the operator $\diag_{i=1}^n(x_i)$ builds a matrix with diagonal entries given by $x_1,\ldots,x_n$ whereas $\col_{i=1}^n(x_i)$ creates a vector by vertically stacking them with $x_1$ on the top. \blue{The $\ell_1$-norm of a (possibly infinite) vector $x$ with elements in $\mathbb{R}$ is defined as $||x||_1:=\sum_i|x_i|$ whereas the $L_1$-norm of a vector-valued function $f:\mathbb{R}_{\ge0}\to\mathbb{R}^n$ is defined as $||f||_{L_1}:=\int_0^\infty\sum_{i=1}^n|f_i(s)|\d s$. We say that $x\in\ell_1$ (resp. $f\in L_1$) if its $\ell_1$-norm (resp. $L_1$-norm) is finite.}

\noindent\textbf{Outline.} The structure of the paper is as follows: in Section \ref{sec:preliminary} preliminary definitions and results are given. Section \ref{sec:main} is devoted to the derivation of design conditions for the considered class of interval observers. Section \ref{sec:switched} presents the results on switched systems. The examples are treated in the related sections.\\

\section{Preliminaries on linear positive impulsive systems}\label{sec:preliminary}

The objective of this section is to recall and extend some existing results about linear positive impulsive systems \cite{Briat:16c} for which we provide new more general proofs. To this aim, let us consider the following class of linear impulsive system \blue{that will be instrumental in designing interval observers for impulsive and switched systems:}
\begin{equation}\label{eq:mainsyst2}
\begin{array}{rcl}
  \blue{\dfrac{\d}{\d\tau}}x(t_k+\tau)&=&A(\tau)x(t_k+\tau)+E_c(\tau)w_c(t_k+\tau), \tau\in(0,T_k],k\in\mathbb{Z}_{\ge 0}\\
  x(t_k^+)&=&Jx(t_k)+E_dw_d(k), k\in\mathbb{Z}_{\ge 1}\\
 \blue{z_c(t_k+\tau)}&\blue{=}&\blue{C_cx(t_k+\tau)+F_cw_c(t_k+\tau)},\tau\in(0,T_k],k\in\mathbb{Z}_{\ge 0}\\
  z_d(k)&=&C_dx(t_k)+F_dw_d(k), k\in\mathbb{Z}_{\ge 1}\\
  x(0)&=&x(0^+)=x_0
\end{array}
\end{equation}
where $x,x_0\in\mathbb{R}^n$, $w_c\in\mathbb{R}^{p_c}$, $w_d\in\mathbb{R}^{p_d}$, $y_c\in\mathbb{R}^{q_c}$ and $y_d\in\mathbb{R}^{q_d}$ are the state of the system, the initial condition, the continuous-time exogenous input, the discrete-time exogenous input, the continuous-time performance output and the discrete-time performance output, respectively. Above, $x(t_k^+):=\lim_{s\downarrow t_k}x(s)$ and the matrix-valued functions $A(\tau)\in\mathbb{R}^{n\times n}$ and $E_c(\tau)\in\mathbb{R}^{n\times p_c}$ are continuous. The sequence of impulse times $\{t_k\}_{k\ge1}$ is assumed to verify the properties: (a) $T_k:=t_{k+1}-t_k>0$ for all $k\in\mathbb{Z}_{\ge0}$ (no jump at $t_0=0$) and (b) $t_k\to\infty$ as $k\to\infty$. When all the above properties hold, the solution of the system \eqref{eq:mainsyst2} exists for all times. \blue{It is worth mentioning that the matrices of the system depends on the timer variable $\tau$. The reason why such a class of systems is considered is that the dynamics of the observation errors takes such a form when considering timer-dependent observer gains.}

\blue{\begin{define}
We define the state-transition matrix of the continuous-time part of the system \eqref{eq:mainsyst2} with $w_c\equiv0$ as
  \begin{equation}\label{eq:Psi}
    \dfrac{\textnormal{d}}{\textnormal{d}\tau}\Psi(\tau,s)=A(\tau)\Psi(\tau,s),\Psi(\tau,\tau)=I_n, 0\le s \le \tau.
  \end{equation}
  This matrix has the following interesting properties:
  \begin{enumerate}[(a)]
    \item The matrix $\Psi(\tau,s)$ is invertible  for all $0\le s\le\tau$ and we have that $\Psi(\tau,s)^{-1}=\Psi(s,\tau)$.
    \item The following equality holds:
    \begin{equation}
           \dfrac{\textnormal{d}}{\textnormal{d}s}\Psi(\tau,s)=-\Psi(\tau,s)A(s).
    \end{equation}
  \end{enumerate}
\end{define}}

\blue{We have the following known result regarding the internal positivity of the impulsive system \eqref{eq:mainsyst2} that will be instrumental in obtaining constructive conditions for the design of interval observers for impulsive and switched systems.} This result can be seen as a generalization of those in \cite{Briat:16c,Briat:17ifacObs} for systems with inputs and outputs. It is therefore stated without proof as it is an adaptation of positivity results for standard linear continuous-time and discrete-time systems:
\begin{proposition}\label{prop:positive}
The following statements are equivalent:
\begin{enumerate}[(a)]
  \item The system \eqref{eq:mainsyst2} is internally positive, i.e. for any $x_0\ge0$, $w_c(t)\ge0$ and $w_d(k)\ge0$, we have that $x(t),z_c(t)\ge0$ for all $t\ge0$ and $z_d(k)\ge0$ for all $k\in\mathbb{Z}_{\ge0}$.
  \item The matrix-valued function $A(\tau)$ is Metzler\footnote{\blue{A square matrix is Metzler if all its off-diagonal entries are nonnegative.}} for all $\tau\ge0$, the matrix-valued function $E_c(\tau)$ is nonnegative\footnote{\blue{A matrix is nonnegative if all its entries are nonnegative.}}  for all $\tau\ge0$ and the matrices $J,E_d,C_c,F_c,C_d,F_d$ are nonnegative.
\end{enumerate}
\end{proposition}

It has been pointed out in \cite{Briat:11g,Ebihara:11,Briat:11h} that the concepts of $L_1$- and $\ell_1$-gains were convenient to use in the context of linear positive systems. As the system \eqref{eq:mainsyst} consists of both continuous-time and discrete-time systems, it seems natural to define the concept of hybrid $L_1/\ell_1$-gain for the hybrid dynamics.
\begin{define}
We say that the system \eqref{eq:mainsyst2} has a hybrid $L_1/\ell_1$-gain of at most $\gamma$ if for all $w_c\in L_1$ and $w_d\in\ell_1$, we have that
  \begin{equation}
    ||z_c||_{L_1}+||z_d||_{\ell_1}<\gamma(||w_c||_{L_1}+||w_d||_{\ell_1})+v(||x_0||)
  \end{equation}
  for some increasing function $v$ verifying $v(0)=0$ and $v(s)\to\infty$ as $s\to\infty$.
\end{define}

It is interesting to note that the concept of hybrid $L_2/\ell_2$-gain has been proposed in \cite{Sivashankar:94} in the context of impulsive dynamics with application to sampled-data systems in order to quantify and optimize the performance of sampled-data controllers. In the present paper, this notation will be analogously used in the context of observer design.

\subsection{Stability under range dwell-time}

The following result provides a sufficient condition for the stability of the system \eqref{eq:mainsyst2} under a range dwell-time constraint; i.e. $T_k\in[\Tmin ,\Tmax  ]$, $k\ge0$, for some given $0<\Tmin \le \Tmax  <\infty$. It is an extension of the range dwell-time result derived in \cite{Briat:16c}.
\blue{\begin{theorem}\label{th:rangeDT}
Let us consider the system \eqref{eq:mainsyst2} with $w_c\equiv0$, $w_d\equiv0$ and assume that it is internally positive; i.e. $A(\tau)$ is Metzler for all $\tau\in[0,\Tmax  ]$ where $0<\Tmin \le \Tmax  <\infty$ are given real numbers. Then, the following statements are equivalent:
\begin{enumerate}[(a)]
  \item There exists a vector $\lambda\in\mathbb{R}_{>0}^n$ such that
  \begin{equation}\label{eq:dksmldksdksmdk}
      \lambda^T\left(J-\Phi(\theta)^{-1}\right)<0
  \end{equation}
  holds for all $\theta\in[\Tmin ,\Tmax  ]$ with $\Phi(\theta)=\Psi(\theta,0)$.
  \item There exist a differentiable vector-valued function $\zeta:[0,\bar T]\mapsto\mathbb{R}^n$, $\zeta(0)>0$, and a scalar $\eps>0$ such that the conditions
 \begin{equation}\label{eq:rangeDT_b}
          \begin{array}{rcl}
            \dot\zeta(\tau)^T+\zeta(\tau)^TA(\tau)&\le&0\\
           \zeta(0)^T J-\zeta(\theta)^T+\eps\mathds{1}^T&\le&0
          \end{array}
        \end{equation}
hold for all $\tau\in[0,\Tmax  ]$ and all $\theta\in[\Tmin ,\Tmax  ]$.
\end{enumerate}
Moreover, when one of the above statements holds, then the positive impulsive system \eqref{eq:mainsyst2} is asymptotically stable under range dwell-time $[\Tmin ,\Tmax  ]$.\hfill\mendth
\end{theorem}
\begin{proof}
  The proof of this result is similar to the one of the corresponding result for LTI positive system in \cite{Briat:16c}. However, since the system is not time-invariant, we have to use state-transition matrices instead of matrix exponentials to prove the equivalence as in \cite{Briat:15f}. Let us prove first that (b) implies (a). The differential inequality \eqref{eq:rangeDT_b} implies that $\zeta(\theta)^T-\zeta(0)^T\Psi(\theta,0)^{-1}\le0$, $\theta\ge0$. Combining this expression with the second one yields $\zeta(0)^T J-\zeta(0)^T\Psi(\theta,0)^{-1}+\eps\mathds{1}^T\le0$, which proves the implication with $\lambda=\zeta(0)>0$.
  %
%

  To prove the converse the goal is to show now that when the conditions of statement (a) hold, then the conditions of statement (b) hold with $\zeta(\tau)=\zeta^*(\tau):=\Psi(\tau,0)^{-T}\lambda$ where $\lambda$ is such that the condition of statement (a) holds. First note that $\zeta^*(0)=\lambda>0$ and that $\dot{\zeta}^*(\tau)=-A(\tau)^T\Psi(\tau,0)^{-T}\lambda=-A(\tau)^T\zeta^*(\tau)$. Therefore, the first condition in \eqref{eq:rangeDT_b} holds with  $\zeta(\tau)=\zeta^*(\tau)$, with the inequality being an equality. Evaluating now the second inequality with  $\zeta(\tau)=\zeta^*(\tau)$ yields $\lambda^TJ-\lambda^T\Psi(\theta,0)^{-1}+\eps\mathds{1}^T\le0$. This proves the result.

\end{proof}}

The first statement in the above result considers a discrete-time Lyapunov function $V(x)=\lambda^Tx$, $\lambda>0$, \blue{or equivalently, $V(x)=\zeta(0)^Tx$, $\zeta(0)>0$}, as in \cite{Briat:16c}. The resulting condition consists of a semi-infinite linear program that is very difficult to check since the matrix $\Phi(\theta)$ does not have any closed-form expression, in general. On the other hand, the second statement demonstrates that an equivalent condition can be stated without the need for this matrix provided that we turn the conditions into an infinite-dimensional linear program (due to the presence of the decision variable $\zeta(\tau)$). This problem can then be relaxed into a finite-dimensional one using a piecewise linear approximation of the decision variable $\zeta(\tau)$ \cite{Allerhand:11,Briat:15f,Xiang:15a,Briat:16c}, or a polynomial approximation using either Putinar's Positivstellensatz \cite{Putinar:93,Parrilo:00,Briat:16c} or Handelman Theorem \cite{Handelman:88,Briat:16c}. The conditions in the last statement of the theorem being affine in the system data make them convenient for design purposes, as also explained in \cite{Briat:13d,Briat:15f,Briat:16c}.

We now extend this result to the case of linear positive impulsive systems with inputs and outputs for which we aim at characterizing their hybrid  $L_1/\ell_1$-gain:
\begin{theorem}\label{th:rangeDTL1}
  Let us consider the system \eqref{eq:mainsyst2} and assume that it is internally positive and let $0<\Tmin \le \Tmax  <\infty$ be given real numbers. Assume that there exist a differentiable vector-valued function $\zeta:[0,\bar T]\mapsto\mathbb{R}^n$, $\zeta(0)>0$, and scalars $\eps,\gamma>0$ such that the conditions
   \begin{equation}\label{eq:rangeDTL1}
          \begin{array}{rcl}
            \dot\zeta(\tau)^T+\zeta(\tau)^TA(\tau)+\mathds{1}^TC_c&\le&0\\
            \zeta(\tau)^TE_c(\tau)+\mathds{1}^TF_c-\gamma\mathds{1}^T&\le&0\\
           \zeta(0)^T J-\zeta(\theta)^T+\mathds{1}^TC_d+\eps\mathds{1}^T&\le&0\\
           \zeta(0)^T E_d+\mathds{1}^TF_d-\gamma\mathds{1}^T&\le&0
          \end{array}
        \end{equation}
        hold for all $\tau\in[0,\Tmax  ]$ and all $\theta\in[\Tmin ,\Tmax  ]$.

        Then, the system \eqref{eq:mainsyst2} is asymptotically stable under range dwell-time $[\Tmin ,\Tmax  ]$ and its hybrid $L_1/\ell_1$-gain is at most $\gamma$.
\end{theorem}
\begin{proof}
The asymptotic stability follows directly from \blue{the first and third inequalities in \eqref{eq:rangeDTL1} and Theorem \ref{th:rangeDT} with the discrete-time Lyapunov function $\zeta(0)^Tx$}. So, let us focus on the gain characterization. To this aim, let us consider the storage function
\begin{equation}
  V(t,x(t))=V_k(t-t_k,x(t)):=\zeta(t-t_k)^Tx(t), t\in(t_k,t_{k+1}], k\in\mathbb{Z}_{\ge0}.
\end{equation}
Using the timer-variable $\tau:=t-t_k$, this can be rewritten as
\begin{equation}
  V(t,x(t))=V_k(\tau,x(t_k+\tau)):=\zeta(\tau)^Tx(t_k+\tau), \tau\in(0,T_k],k\in\mathbb{Z}_{\ge0}.
\end{equation}
\blue{Multiplying to the right the two first inequalities in \eqref{eq:rangeDTL1} by $x(t_k+\tau)$ and $w_c(t_k+\tau)$, respectively, and summing them yield
\begin{equation}
  \left[\dot\zeta(\tau)^T+\zeta(\tau)^TA(\tau)+\mathds{1}^TC_c\right]x(t_k+\tau)+\left[
            \zeta(\tau)^TE_c(\tau)+\mathds{1}^TF_c-\gamma\mathds{1}^T\right]w_c(t_k+\tau)\le0.
\end{equation}
This inequality can be rewritten as
\begin{equation}
  \dfrac{\d}{\d\tau}V_k(\tau,x(t_k+\tau))+\mathds{1}^Tz_c(t_k+\tau)-\gamma\mathds{1}^Tw_c(t_k+\tau)\le0.
\end{equation}
Integrating this expression from $0$ to $T_k$ yields
  \begin{equation}\label{eq:proof1}
    V_k(T_k,x(t_{k+1}))-V_k(0,x(t_k^+))+S_k^c\le0
  \end{equation}
  where
  \begin{equation}
    S_k^c:=\int_{t_k}^{t_{k+1}}\left[\mathds{1}^Tz_c(s)-\gamma\mathds{1}^Tw_c(s)\right]\d s.
  \end{equation}
  Similarly, multiplying to the right the two last inequalities in \eqref{eq:rangeDTL1} by $x(t_{k+1})$ and $w_d(k+1)$, respectively, letting $\theta=T_k$ and summing them yields
  \begin{equation}
      [\zeta(0)^T J-\zeta(T_k)^T+\mathds{1}^TC_d+\eps\mathds{1}^T]x(t_{k+1})+[\zeta(0)^T E_d+\mathds{1}^TF_d-\gamma\mathds{1}^T]w_d(k+1)\le0.
  \end{equation}
  This can be reformulated as
  \begin{equation}
    \zeta(0)^Tx(t_{k+1}^+)-\zeta(T_k)^Tx(t_{k+1})+\mathds{1}^Tz_c(t_{k+1})-\gamma\mathds{1}^Tw_d(k+1)\le-\eps\mathds{1}^Tx(t_{k+1})
  \end{equation}
  or, equivalently,
  \begin{equation}\label{eq:proof2}
   V _{k+1}(0,x(t_{k+1}^+))-V_k(T_k,x(t_{k+1}))+ S_{k+1}^d\le-\eps\mathds{1}^Tx(t_{k+1})
    \end{equation}
        where
      \begin{equation}
    S_k^d:=\mathds{1}^Tz_d(k)-\gamma\mathds{1}^Tw_d(k).
  \end{equation}
  Adding the inequalities \eqref{eq:proof1}-\eqref{eq:proof2} together yields
  \begin{equation}
        V_{k+1}(0,x(t_{k+1}^+))-V_k(0,x(t_k^+))+S_k^c+S_{k+1}^d\le-\eps\mathds{1}^Tx(t_{k+1}).
  \end{equation}
Since the system is asymptotically stable and $w_c\in L_1$, $w_d\in\ell_1$, then $V_{k+1}(0,x(t_{k+1}^+))\to0$ as $k\to\infty$. Hence, summing over $k$ the above inequality yields
\begin{equation}
  -V_0(0,x(0))+\sum_{k=0}^\infty\left[S_k^c+S_{k+1}^d\right]<-\eps\sum_{k=0}^\infty \mathds{1}^Tx(t_{k+1}).
\end{equation}
Note that the term $\sum_{k=0}^\infty \mathds{1}^Tx(t_{k+1})$ is finite since the discrete-time part of the system is $\ell_1$-stable (from the asymptotic stability and the linearity of the system). This can be rewritten as
\begin{equation}
\begin{array}{rcl}
   ||z_c||_{L_1}+||z_d||_{\ell_1}&\le &\gamma(||w_c||_{L_1}+||w_d||_{\ell_1})+\zeta(0)^Tx(0)-\eps\displaystyle \sum_{k=0}^\infty \mathds{1}^Tx(t_{k+1})\\
   &<&\gamma(||w_c||_{L_1}+||w_d||_{\ell_1})+\zeta(0)^Tx(0).
\end{array}
\end{equation}
which proves the result.}
\end{proof}


\subsection{Stability under minimum dwell-time}

The following result provides a sufficient condition for the stability of the system \eqref{eq:mainsyst2} under a minimum dwell-time constraint; i.e. $T_k\ge\bar T$, $k\ge0$, for some given $\bar T>0$. It is an extension of the minimum dwell-time result derived in \cite{Briat:16c} for which we also provide a more general proof:
\blue{\begin{theorem}\label{th:minDT}
Let us consider the system  \eqref{eq:mainsyst2} with $w_c\equiv0$, $w_d\equiv0$, $A(\tau)=A(\bar{T})$ for all $\tau\ge\bar{T}>0$, where $\bar T>0$ is given, and assume that it is internally positive. Then, the following statements are equivalent:
\begin{enumerate}[(a)]
  \item There exists a vector $\lambda\in\mathbb{R}_{>0}^n$ such that
  \begin{equation}
      \lambda^TA(\bar{T})<0
  \end{equation}
  and
    \begin{equation}
      \lambda^T\left(\Phi(\bar{T})J-I_n\right)<0
  \end{equation}
  hold  where $\Phi(\bar{T}=\Psi(\bar{T},0)$.
  \item There exist a differentiable vector-valued function $\zeta:[0,\bar T]\mapsto\mathbb{R}^n$, $\zeta(\bar T)>0$, and a scalar $\eps>0$ such that the conditions
 \begin{equation}\label{eq:minDT_b}
          \begin{array}{rcl}
           \zeta(\bar T)^TA(\bar{T})+\eps\mathds{1}^T&\le&0\\
           \dot\zeta(\tau)^T+\zeta(\tau)^TA(\tau)&\le&0\\
           \zeta(0)^T J-\zeta(\bar T)^T+\eps\mathds{1}^T&\le&0
          \end{array}
        \end{equation}
hold for all $\tau\in[0,\bar T]$.
\end{enumerate}
Moreover, when one of the above statements holds, then the positive impulsive system \eqref{eq:mainsyst2} is asymptotically stable under minimum dwell-time $\bar T$.\hfill\mendth
\end{theorem}
\begin{proof}
  The proof that the feasibility of the conditions in statement (a) imply that the system \eqref{eq:mainsyst2} is stable under minimum dwell-time $\bar T$ can be found in \cite{Briat:16c}. We now prove that the feasibility of the conditions of statement (b) imply that of the conditions in statement (a). The second inequality in statement (b) implies, as previously, that $\zeta(\bar T)^T-\zeta(0)^T\Psi(0,\bar{T})\le0$. Multiplying then this expression from the right by $\Psi(\bar{T},0)J\ge0$ yields $\zeta(\bar T)^T\Psi(\bar T,0)J-\zeta(0)^TJ\le0$. Combining this inequality with the third one of statement (b), we finally obtain that $\zeta(\bar T)^T\Phi(\bar T) J-\zeta(\bar T)^T+\eps\mathds{1}^T\le0$, which proves the implication.



  To prove the converse, let us define the vector-valued function $\zeta^*(\tau)^T:=\lambda^T\Psi(\bar T,\tau)$ where $\lambda>0$ is such that the condition in the statement (a) holds. Note that $\Phi(\bar T)=\Psi(\bar T,0)$. We show now that if the conditions of statement (a) hold, then the conditions of statement (b) hold with $\zeta(\tau)=\zeta^*(\tau)$. First note that $\zeta^*(\bar T)=\lambda>0$, $\zeta^*(0)=\Phi(\bar T)^T\lambda$ and that $\dot{\zeta}^*(\tau)^T=-\lambda^T\Psi(\bar{T},\tau)A(\tau)=-\zeta^*(\tau)^TA(\tau)$. Hence, the first condition in \eqref{eq:minDT_b} holds with  $\zeta(\tau)=\zeta^*(\tau)$, with the inequality being an equality. Evaluating now the second inequality with  $\zeta(\tau)=\zeta^*(\tau)$ yields $\lambda^T\Phi(\bar T)J-\lambda^T+\eps\mathds{1}^T\le0$. This proves the result.
\end{proof}}

As in the range dwell-time case, we extend this result to account for the input/output behavior of the system \eqref{eq:mainsyst2} in a hybrid $L_1/\ell_1$ sense.
\begin{theorem}\label{th:minDTL1}
Let us consider the system  \eqref{eq:mainsyst2} with $A(\tau)=A(\bar{T})$ and $E_c(\tau)=E_c(\bar{T})$ for all $\tau\ge\bar{T}>0$, where $\bar T>0$ is given, and assume that it is internally positive. Assume further that there exist a differentiable vector-valued function $\zeta:[0,\bar T]\mapsto\mathbb{R}^n$, $\zeta(\bar T)>0$, and scalars $\eps,\gamma>0$ such that the conditions
 \begin{equation}
          \begin{array}{rcl}
           \zeta(\bar T)^TA(\bar{T})+\zeta(\bar T)^TE_c(\bar{T})+\eps\mathds{1}^T&\le&0\\
            \zeta(\bar T)^TE_c(\bar T)+\mathds{1}^TF_c-\gamma\mathds{1}^T&\le&0\\
              \dot\zeta(\tau)^T+\zeta(\tau)^TA(\tau)+\mathds{1}^TC_c&\le&0\\
            \zeta(\tau)^TE_c(\tau)+\mathds{1}^TF_c-\gamma\mathds{1}^T&\le&0\\
           \zeta(0)^T J-\zeta(\bar T)^T+\mathds{1}^TC_d+\eps\mathds{1}^T&\le&0\\
           \zeta(0)^T E_d+\mathds{1}^TF_d-\gamma\mathds{1}^T&\le&0
          \end{array}
        \end{equation}
hold for all $\tau\in[0,\bar T]$. Then, the positive impulsive system \eqref{eq:mainsyst2} is asymptotically stable under minimum dwell-time $\bar T$ and its hybrid $L_1/\ell_1$-gain is at most $\gamma$.\hfill\mendth
\end{theorem}
\begin{proof}
The proof of this result follows from the same arguments as for the proof of Theorem \ref{th:rangeDTL1} and is thus omitted.
\end{proof}

\blue{\begin{remark}
  In the above results, we may question the relevance or the limitations of the assumption that  $A(\tau)=A(\bar{T})$ and $E_c(\tau)=E_c(\bar{T})$ for all $\tau\ge\bar{T}>0$. In fact, this assumption is not restrictive at all for the problems considered in this paper. As shown in many papers on control/observation for switched and impulsive systems under minimum dwell-time, the closed-loop system or the observation error verifies this assumption through a natural choice for the gains of the controller/observer; see e.g. \cite{Allerhand:11,Briat:13d,Briat:14f,Briat:15f,Briat:15i} and the references therein. This natural choice arises from the aim of obtaining convex stabization/observation conditions that can be checked in an efficient way.
\end{remark}}

\section{$L_1/\ell_1$-to-$L_1/\ell_1$ interval observation of linear impulsive systems}\label{sec:main}

\blue{Let us now consider the main problem of the paper, that is, the derivation of interval observers for linear impulsive systems. The system \eqref{eq:mainsyst2} was defined in order to formulate relevant stability conditions for linear positive impulsive systems. However, we shall derive interval observers for the following class of linear impulsive system\footnote{\blue{Note that it is straightforward to extend the scope of the paper by making all the matrices of the system \eqref{eq:mainsyst} timer-dependent. All the subsequent results will still hold through very minor changes. This is not done in this paper for the sake of simplicity.}}:}
\begin{equation}\label{eq:mainsyst}
\begin{array}{rcl}
  \dot{x}(t)&=&Ax(t)+E_cw_c(t),t\notin\{t_k\}_{k\in\mathbb{Z}_{\ge1}}\\
  x(t_k^+)&=&Jx(t_k)+E_dw_d(k),k\in\mathbb{Z}_{\ge1}\\
  y_c(t)&=&C_{c}x(t)+F_cw_c(t),t\in\mathbb{R}_{\ge0}\\
  y_d(k)&=&C_{d}x(t_k)+F_dw_d(k),k\in\mathbb{Z}_{\ge0}\\
  x(t_0)&=&x(t_0^+)=x_0
\end{array}
\end{equation}
where $x,x_0\in\mathbb{R}^n$, $w_c\in\mathbb{R}^{p_c}$, $w_d\in\mathbb{R}^{p_d}$, $y_c\in\mathbb{R}^{q_c}$ and $y_d\in\mathbb{R}^{q_d}$ are the state of the system, the initial condition, the continuous-time exogenous input, the discrete-time exogenous input, the continuous-time measured output and the discrete-time measured output, respectively. The sequence of impulse instants $\{t_k\}_{k\in\mathbb{Z}_{\ge0}}$ (no jump again at $t_0$) is assumed to satisfy the same properties as for the system \eqref{eq:mainsyst2}. The input signals are all assumed to be bounded functions and that some bounds are known; i.e. we have $w_c^-(t)\le w_c(t)\le w_c^+(t)$ and $w_d^-(k)\le w_d(k)\le w_d^+(k)$ for all $t\ge0$ and $k\ge0$ and for some known $w_c^-(t), w_c^+(t),w_d^-(k),w_d^+(k)$.

\subsection{Proposed interval observer}

We are interested in finding an interval-observer of the following form \blue{for the system \eqref{eq:mainsyst}:}
\begin{equation}\label{eq:obs}
\begin{array}{lcl}
      \dot{x}^\bullet(t)&=&Ax^\bullet(t)+E_c w_c^\bullet(t)+L_c(t)(y_c(t)-C_c x^\bullet(t)-F_c w_c^\bullet(t))\\
      x^\bullet(t_k^+)&=&Jx^\bullet(t_k)+E_d w_d^\bullet(t)+L_d(y_d(k)-C_d x^\bullet(t_k)-F_d w_d^\bullet(t))\\
      x^\bullet(0)&=&x_0^\bullet
\end{array}
\end{equation}
where $\bullet\in\{-,+\}$. Above, the observer with the superscript ``$+$'' is meant to estimate an upper-bound on the state value whereas the observer with the superscript ``-'' is meant to estimate a lower-bound, i.e. $x^-(t)\le x(t)\le x^+(t)$ for all $t\ge0$ provided that $x_0^-\le x_0\le x_0^+$, $w_c^-(t)\le w_c(t)\le w_c^+(t)$ and $w_d^-(k)\le w_d(k)\le w_d^+(k)$. The errors dynamics $e^+(t):=x^+(t)-x(t)$ and $e^-(t):=x(t)-x^-(t)$ are then described by
\begin{equation}\label{eq:error}
\begin{array}{rcl}
    \dot{e}^\bullet(t)&=&(A-L_c(t)C_c)e^\bullet(t)+(E_c-L_c(t) F_c)\delta_c^\bullet(t)\\ 
    e^\bullet(t_k^+)&=&(J-L_d C_d)e^\bullet(t_k)+(E_d-L_d F_d)\delta_d^\bullet(k)\\
    e_c^\bullet(t)&=&M_ce^\bullet(t)\\
    e_d^\bullet(k)&=&M_de^\bullet(t_k)
\end{array}
\end{equation}
where $\bullet\in\{-,+\}$, $\delta_c^+(t):=w_c^+(t)-w_c(t)\in\mathbb{R}_{\ge0}^{p_c}$, $\delta_c^-(t):=w_c(t)-w_c^-(t)\in\mathbb{R}_{\ge0}^{p_c}$, $\delta_d^+(k):=w_d^+(k)-w_d(k)\in\mathbb{R}_{\ge0}^{p_d}$ and $\delta_d^-(k):=w_d(k)-w_d^-(k)\in\mathbb{R}_{\ge0}^{p_d}$. The continuous-time and discrete-time performance outputs are denoted by  $e_c^\bullet(t)$ and $e_d^\bullet(k)$, respectively. Note that both errors have exactly the same dynamics and, consequently, it is unnecessary here to consider different observer gains. Note that this would not be the case if the observers were coupled in a non-symmetric way. The matrices $M_c,M_d\in\mathbb{R}^{n\times n}_{\ge0}$ are nonzero weighting matrices that are needed to be chosen a priori.

\subsection{Range dwell-time result}

In the range-dwell -time case, the time-varying gain $L_c(t)$ in \eqref{eq:obs} is defined as follows
\begin{equation}\label{eq:L1}
  L_c(t)=\tilde{L}_c(t-t_k),\ t\in(t_k,t_{k+1}]
\end{equation}
where $\tilde{L}_c:[0,\Tmax  ]\mapsto\mathbb{R}^{n\times q_c}$ is a matrix-valued function to be determined. The rationale for considering such structure is to allow for the derivation of convex design conditions. \blue{When such gains are considered, the dynamics of the error \eqref{eq:error}-\eqref{eq:L1} has the same structure as the system \eqref{eq:mainsyst2} with $\tau:=t-t_k$ and, therefore, the results in Section \ref{sec:preliminary} can be used.}

The interval observation problem is defined, in this case, as follows:
\begin{problem}\label{problem1}
Find an interval observer of the form \eqref{eq:obs} (i.e. a matrix-valued function $L_c(\cdot)$ of the form \eqref{eq:L1} and a matrix $L_d\in\mathbb{R}^{n\times q_d}$) such that the error dynamics \eqref{eq:error} is
  \begin{enumerate}[(a)]
    \item state-positive, that is
    \begin{itemize}
      \item $A-\tilde L_c(\tau)  C_c$ is Metzler for all $\tau\in[0,\Tmax  ]$,
      \item $E_c-\tilde L_c(\tau)  F_c $ is nonnegative for all $\tau\in[0,\Tmax  ]$,
      \item $J-L_dC_d$ and $E_d-L_dF_d$ are nonnegative,
    \end{itemize}
    \item asymptotically stable under range dwell-time $[\Tmin ,\Tmax  ]$ when $w_c\equiv0$ and $w_d\equiv0$, and
    \item satisfies the condition
    \begin{equation}
    ||e_c^\bullet||_{L_1}+||e_d^\bullet||_{\ell_1}<\gamma(||\delta^\bullet_c||_{L_1}+||\delta^\bullet_d||_{\ell_1}),\ \bullet\in\{-,+\}
  \end{equation}
  for some $\gamma$, for all $w_c\in L_1$ and all $w_d\in\ell_1$ when $x_0=0$.
  \end{enumerate}
\end{problem}

\blue{\begin{remark}
  It is  important to mention here that the observer structure \eqref{eq:obs} is rather simple and, hence, restrictive in nature. In this regard, there may not exist an interval-observer of this form which solves the above interval-observation problem. Fortunately, the approach proposed in this paper is general enough to be readily adaptable to more complex interval-observer structures. For instance, one may want to find an invertible matrix $P$ for which the change of variables $\tilde{e}=Pe$ yields a new dynamical system for which the positivity conditions in Problem \ref{problem1} are more easily satisfied. Computing such a matrix is not easy but some methods exist; see e.g. \cite{Raissi:12,Chambon:15}. Other interval-observer structures \cite{Raissi:18} splitting certain matrices into a positive and a negative part can also be easily considered using the proposed approach as the obtained stability conditions are very general.
\end{remark}}

The following result provides a sufficient condition for the solvability of Problem \ref{problem1}:
\begin{theorem}\label{th:1}
Assume that there exist a differentiable matrix-valued function $X:[0,\Tmax  ]\mapsto\mathbb{D}^n$, $X(0)\succ0$, a matrix-valued function $U_c:[0,\Tmax  ]\mapsto\mathbb{R}^{n\times q_c}$, a matrix $U_d\in\mathbb{R}^{n\times q_d}$ and scalars $\eps,\alpha,\gamma>0$ such that the conditions
\begin{subequations}\label{eq:th1a}
\begin{alignat}{4}
            X(\tau)A-U_c(\tau)C_c+\alpha I_n&\ge0\label{eq:th1:1}\\
           X(0) J-U_d C_d&\ge0\label{eq:th1:2}\\
           X(\tau)E_c-U_c(\tau)F_c&\ge0\label{eq:th1:3}\\
            X(0)E_d-U_d F_d&\ge0\label{eq:th1:4}
  \end{alignat}
\end{subequations}
           and
 \begin{subequations}\label{eq:th1b}
\begin{alignat}{4}
           \mathds{1}^T\left[\dot{X}(\tau)+X(\tau)A-U_c(\tau)C_c\right]+\mathds{1}^TM_c&\le0\label{eq:th1:5}\\
            \mathds{1}^T\left[X(0) J-U_d C_d-X(\theta)+\eps I\right]+\mathds{1}^TM_d&\le0\label{eq:th1:6}\\
            \mathds{1}^T\left[X(\tau)E_c-U_c(\tau)F_c\right]-\gamma \mathds{1}^T&\le0\label{eq:th1:7}\\
            \mathds{1}^T\left[X(0)E_d-U_d F_d\right]-\gamma \mathds{1}^T&\le0\label{eq:th1:8}
             \end{alignat}
\end{subequations}
%
%
        %
hold for all $\tau\in[0,\Tmax  ]$ and all $\theta\in[\Tmin ,\Tmax  ]$. Then, there exists an interval observer of the form \eqref{eq:obs}-\eqref{eq:L1} that solves Problem \ref{problem1} and suitable observer gains are given by
\begin{equation}\label{eq:formula1}
  \tilde{L}_c(\tau)= X(\tau)^{-1}U_c(\tau)\quad \textnormal{and}\quad L_d=X(0)^{-1}U_d.
\end{equation}
\end{theorem}
\begin{proof}
From the diagonal structure of the matrix-valued function $X(\cdot)$, the changes of variables \eqref{eq:formula1} and Proposition \ref{prop:positive}, we can observe that the inequalities \eqref{eq:th1:1} to \eqref{eq:th1:4} are readily equivalent to saying that the statement (a) of Problem \ref{problem1} holds. Using now the changes of variables $\lambda(\tau)=X(\tau)\mathds{1}$ and \eqref{eq:formula1}, we get that the feasibility of \eqref{eq:th1:5}-\eqref{eq:th1:6} is equivalent to saying that the error dynamics \eqref{eq:error} with \eqref{eq:L1} verifies the range dwell-time conditions of Theorem \ref{th:rangeDTL1} with the same $\lambda(\tau)$. The proof is completed.
\end{proof}

\subsection{Minimum dwell-time result}

In the minimum dwell-time case, the time-varying gain  $L_c$ is defined as follows
\begin{equation}\label{eq:L2}
  L_c(t)=\left\{\begin{array}{ll}
  \tilde{L}_c(t-t_k)& \textnormal{if }t\in(t_k,t_{k}+\tau]\\
  \tilde{L}_c(\bar T)& \textnormal{if }t\in(t_k+\bar{T},t_{k+1}]
  \end{array}\right.
\end{equation}
where $\tilde{L}_c:\mathbb{R}_{\ge0}\mapsto\mathbb{R}^{n\times q_c}$ is a function to be determined. \blue{When such gains are considered, the dynamics of the error \eqref{eq:error}-\eqref{eq:L1} has the same structure as the system \eqref{eq:mainsyst2} with $\tau:=t-t_k$ and, therefore, the results in Section \ref{sec:preliminary} can be used.}

The observation problem is defined, in this case, as follows:
\begin{problem}\label{problem2}
Find an interval observer of the form \eqref{eq:obs} (i.e. a matrix-valued function $L_c(\cdot)$ of the form \eqref{eq:L2} and a matrix $L_d\in\mathbb{R}^{n\times q_d}$) such that the error dynamics \eqref{eq:error} is
  \begin{enumerate}[(a)]
    \item state-positive, that is
    \begin{itemize}
      \item $A-\tilde L_c(\tau) C_c$ is Metzler for all $\tau\in[0,\bar{T}]$,
      \item $E_c-\tilde L_c(\tau) F_c $ is nonnegative for all $\tau\in[0,\bar{T}]$,
      \item $J-L_dC_d$ and $E_d-L_dF_d$ are nonnegative,
    \end{itemize}
    \item asymptotically stable under minimum dwell-time $\bar T$ when $w_c\equiv0$ and $w_d\equiv0$, and
    \item satisfies the condition
    \begin{equation}
    ||e_c^\bullet||_{L_1}+||e_d^\bullet||_{\ell_1}<\gamma(||\delta^\bullet_c||_{L_1}+||\delta^\bullet_d||_{\ell_1}),\ \bullet\in\{-,+\}
  \end{equation}
  for some $\gamma$, for all $w_c\in L_1$ and for all $w_d\in\ell_1$ when $x_0=0$.
  \end{enumerate}
\end{problem}

The following result provides a sufficient condition for the solvability of Problem \ref{problem2}:
\begin{theorem}\label{th:2}
There exists a differentiable matrix-valued function $X:[0,\bar T]\mapsto\mathbb{D}^n$, $X(\bar T)\succ0$, a matrix-valued function $U_c:[0,\bar T]\mapsto\mathbb{R}^{n\times q_c}$, a matrix $U_d\in\mathbb{R}^{n\times q_d}$ and scalars $\eps,\alpha,\gamma>0$ such that the conditions
\begin{subequations}\label{eq:th2a}
\begin{alignat}{4}
           X(\tau)A-U_c(\tau)C_c+\alpha I_n&\ge0\\
           X(\bar T) J-U_d C_d&\ge0\\
           X(\tau)E_c-U_c(\tau)F_c&\ge0\\
           X(\bar T)E_d-U_d F_d&\ge0
  \end{alignat}
\end{subequations}
and
    \begin{subequations}\label{eq:th2b}
\begin{alignat}{4}
           \mathds{1}^T\left[X(\bar T)A-U_c(\bar T)C_c+\eps I_n\right]&\le0\\
           \mathds{1}^T\left[X(\bar T)E_c-U_c(\bar T)F_c\right]-\gamma \mathds{1}^T&\le0\\
           \mathds{1}^T\left[\dot{X}(\tau)+X(\tau)A-U_c(\tau)C_c\right]&\le0\\
           \mathds{1}^T\left[X(\tau)E_c-U_c(\tau)F_c\right]-\gamma \mathds{1}^T&\le0\\
            \mathds{1}^T\left[X(0) J-U_d C_d-X(\bar T)+\eps I\right]&\le0\\
            \mathds{1}^T\left[X(0)E_d-U_d F_d\right]-\gamma \mathds{1}^T&\le0
\end{alignat}
\end{subequations}
hold for all $\tau\in[0,\bar T]$.  Then, there exists an interval observer of the form \eqref{eq:obs}-\eqref{eq:L2} that solves Problem \ref{problem2} and suitable observer gains are given by
\begin{equation}
  \tilde{L}_c(\tau)= X(\tau)^{-1}U_c(\tau)\quad \textnormal{and}\quad L_d=X(\bar{T})^{-1}U_d.
\end{equation}
\end{theorem}
\begin{proof}
The proof is similar to that of Theorem \ref{th:1} and is omitted.
\end{proof}

\subsection{Computational considerations}

Several methods can be used to check the conditions stated in Theorem \ref{th:1} and Theorem \ref{th:2}. The piecewise linear discretization approach \cite{Allerhand:11,Xiang:15a,Briat:16c} assumes that the decision variables are piecewise linear functions of their arguments and leads to a finite-dimensional linear program that can be checked using standard linear programming algorithms. Another possible approach is based on Handelman's Theorem \cite{Handelman:88} and also leads to a finite-dimensional program \cite{Briat:11h,Briat:16c}. We opt here for an approach based on Putinar's Positivstellensatz \cite{Putinar:93} and semidefinite programming \cite{Parrilo:00}\footnote{See \cite{Briat:16c} for a comparison of all these methods.}. Before stating the main result of the section, we need to define first some terminology. A multivariate polynomial $p(x)$ is said to be a sum-of-squares (SOS) polynomial if it can be written as $\textstyle p(x)=\sum_{i}q_i(x)^2$ for some polynomials $q_i(x)$. A polynomial matrix $p(x)\in\mathbb{R}^{n\times m}$ is said to \emph{componentwise sum-of-squares} (CSOS) if each of its entries is an SOS polynomial. Checking whether a polynomial is SOS can be exactly cast as a semidefinite program \cite{Parrilo:00,Chesi:10b} that can be easily solved using semidefinite programming solvers such as SeDuMi \cite{Sturm:01a}. The package SOSTOOLS \cite{sostools3} can be used to formulate and solve SOS programs in a convenient way.

Below is the SOS implementation of the conditions of Theorem \ref{th:1}:
\begin{proposition}\label{prop:SOS1}
  Let $d\in\mathbb{N}$, $\eps>0$ and $\epsilon>0$ be given and assume that there exist polynomials $\chi_i:\mathbb{R}\mapsto\mathbb{R}$, $i=1,\ldots,n$, $\Psi_1:\mathbb{R}\mapsto\mathbb{R}^{n\times n}$,  $\Psi_2:\mathbb{R}\mapsto\mathbb{R}^{n\times q_c}$ and $\psi_1,\psi_2:\mathbb{R}\mapsto\mathbb{R}^{n}$, $\psi_3:\mathbb{R}\mapsto\mathbb{R}^{p_c}$ of degree $2d$, a matrix $U_d\in\mathbb{R}^{n\times q_d}$ and scalars $\alpha,\gamma\ge0$ such that
  \begin{enumerate}[(a)]
    \item $\Psi_1(\tau),\Psi_2(\tau),\psi_1(\tau),\psi_2(\tau)$ and $\psi_3(\tau)$ are CSOS,
    \item $X(0)-\epsilon I_n\ge0$ (or is CSOS),
    \item $X(\tau)A-U_c(\tau)C_c+\alpha I_n-\Psi_1(\tau)f(\tau)$ is CSOS,
    \item $X(0) J-U_d C_d\ge0$ (or is CSOS),
    \item $X(\tau)E_c-U_c(\tau)F_c-\Psi_2(\tau)f(\tau)$ is CSOS,
    \item $X(0)E_d-U_d F_d\ge0$ (or is CSOS),
    \item $-\mathds{1}^T\left[\dot{X}(\tau)+X(\tau)A-U_c(\tau)C_c\right]-\mathds{1}^TM_c-f(\tau)\psi_1(\tau)^T$

    is CSOS,
    \item $-\mathds{1}^T\left[X(0) J-U_d C_d-X(\theta)+\eps I\right]-\mathds{1}^TM_d-g(\theta)\psi_2(\theta)^T$

     is CSOS,
     %
     %
     \item $-\mathds{1}^T\left[X(\tau)E_c-U_c(\tau)F_c\right]+\gamma \mathds{1}^T-f(\tau)\psi_3(\tau)$ is CSOS,
     \item $-\mathds{1}^T\left[X(0)E_d-U_d F_d\right]+\gamma \mathds{1}^T\ge0$  (or is CSOS)
  \end{enumerate}
  where $X(\tau):=\diag_{i=1}^n(\chi_i(\tau))$, $f(\tau):=\tau(\Tmax  -\tau)$ and $g(\theta):=(\theta-\Tmin )(\Tmax  -\theta)$.

  Then, the conditions  of Theorem \ref{th:1} hold with the same $X(\tau)$, $U_c(\tau)$, $U_d$, $\alpha$, $\eps$ and $\gamma$.
\end{proposition}
\begin{proof}
The proof follows from the same arguments as the proof of Proposition 3.15 in \cite{Briat:16c}.
\end{proof}

\begin{remark}[Asymptotic exactness]
The above relaxation is asymptotically exact under very mild conditions \cite{Putinar:93} in the sense that if the original conditions of Theorem \ref{th:1} hold then we can find a degree $d$ for the polynomial variables for which the conditions in Proposition \ref{prop:SOS1} are feasible. See \cite{Briat:16c} for more details.
\end{remark}

\subsection{Examples}\label{sec:examples}

All the computations are performed on a computer equipped with a processor i7-5600U@2.60GHz with 16GB of RAM. The conditions are implemented using SOSTOOLS \cite{sostools3} and solved with SeDuMi \cite{Sturm:01a}.

\blue{\subsubsection{Example 1. Range dwell-time}

Let us consider now the system \eqref{eq:mainsyst} with matrices
\begin{equation}\label{eq:ex2}
\begin{array}{l}
    A=\begin{bmatrix}
    -2 & 1\\
    0 & 1
  \end{bmatrix}, E_c=\begin{bmatrix}
    0.1\\
    0.1
  \end{bmatrix}, J=\begin{bmatrix}
1.1 & 1\\
0 & -0.1
  \end{bmatrix},  E_d=\begin{bmatrix}
    0.3\\
    0.3
  \end{bmatrix},\\
    C_c=C_d=\begin{bmatrix}
    0 & 1
  \end{bmatrix}, F_c=F_d=0.1.
\end{array}
\end{equation}
Define also $w_c(t)=\sin(t)$, $w^-(t)=-1$, $w^+(t)=1$, $w_d(k)$ is a stationary random process that follows the uniform distribution $\mathcal{U}(-0.5,0.5)$, $w_d^-=-0.5$ and $w_d^+=0.5$. Using polynomials of degree 2 and solving for the conditions of Theorem \ref{th:rangeDTL1} with $\Tmin =0.3$ and $\Tmax  =0.5$, we get the value $\gamma=.86233$ as minimum together with the observer gains
\begin{equation}\label{eq:Lex1}
  L_d=\begin{bmatrix}
     1\\
     -0.1
  \end{bmatrix}\quad \textnormal{and}\quad \tilde L_c(\tau)=\begin{bmatrix}
    1\\
    1
  \end{bmatrix}.
\end{equation}
Note that the gain $L_c$ is constant and has been obtained from an approximation of the $\tau$-dependent gain which deviates from a very small amount from the above value. For information, the semidefinite program has 167 primal variables, 61 dual variables and it takes 5.57 seconds to solve. To illustrate this result, we generate random impulse times satisfying the range dwell-time condition and we obtain the trajectories depicted in Fig.~\ref{fig:states_rangeDT}. The disturbance inputs are depicted in Fig.~\ref{fig:inputs_rangeDT}.
\begin{figure}
  \centering
  \includegraphics[width=0.8\textwidth]{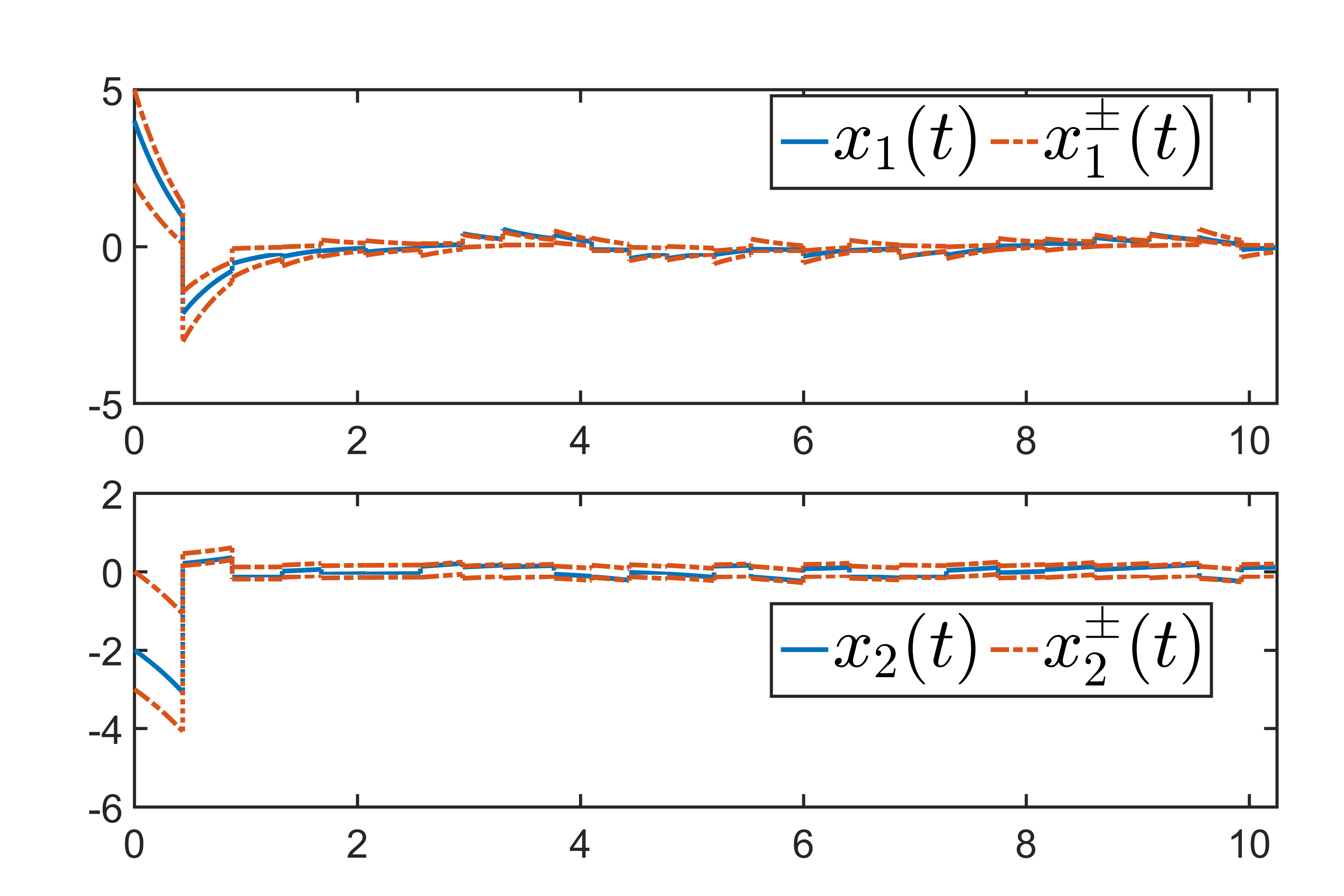}
  \caption{Trajectories of the system \eqref{eq:mainsyst}-\eqref{eq:ex2} and the interval observer \eqref{eq:obs}-\eqref{eq:L1}-\eqref{eq:Lex1} for some randomly chosen impulse times satisfying the range dwell-time $[0.3,\ 0.5]$.}\label{fig:states_rangeDT}
\end{figure}

\begin{figure}
  \centering
  \includegraphics[width=0.8\textwidth]{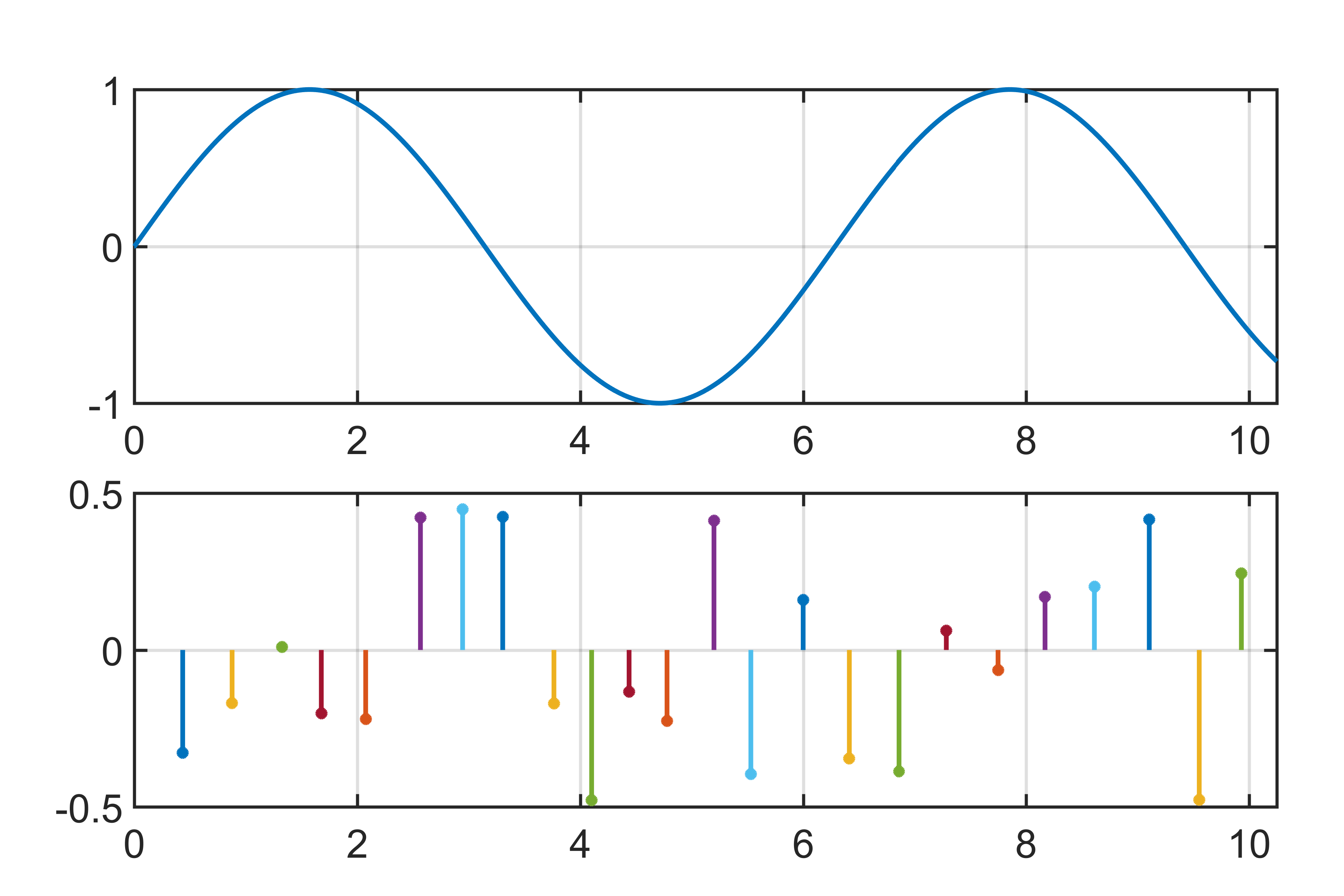}
  \caption{Trajectory of the continuous-time input $w_c$ (top) and the discrete-time input $w_d$ (bottom)}\label{fig:inputs_rangeDT}
\end{figure}

\subsubsection{Example 2. Minimum dwell-time}

Let us consider here the example from \cite{Briat:13d} to which we add disturbances as also done in \cite{Degue:16nolcos}. The matrices of the system are given by
\begin{equation}\label{eq:ex1}
\begin{array}{l}
    A=\begin{bmatrix}
    -1 & 0\\
    1 & -2
  \end{bmatrix}, E_c=\begin{bmatrix}
    0.1\\
    0.1
  \end{bmatrix}, J=\begin{bmatrix}
2 & -1\\
1 & 3
  \end{bmatrix},  E_d=\begin{bmatrix}
    0.3\\
    0.3
  \end{bmatrix},\\
    C_c=C_d=\begin{bmatrix}
    0 & 1
  \end{bmatrix}, F_c=F_d=0.03.
\end{array}
\end{equation}
The disturbances are defined in the same way as in the previous example. Choosing a desired minimum dwell-time of $\bar T=1$ and solving the conditions of Theorem \ref{th:2} with polynomials of degree 4 yield the minimum value 1.6437 for $\gamma$ together with the observer gains
\begin{equation}\label{eq:Lex2}
  L_d=\begin{bmatrix}
    -1\\
     3
  \end{bmatrix}
\end{equation}
and
\begin{equation}
  \tilde L_c(\tau)=\begin{bmatrix}
  0\\
  \dfrac { 0.66556- 0.3359\tau+ 3.0052{\tau}^{2}- 5.8123{\tau}^{3}+ 5.811{\tau}^{4}}{ 0.19967+ 0.025583\tau+ 0.15592{\tau}^{2}- 0.36554{\tau}^{3}+ 0.88447{\tau}^{4}}
  \end{bmatrix}.
\end{equation}
For information, the semidefinite program has 275 primal variables, 75 dual variables and it takes 2.530 seconds to solve. To illustrate this result, we generate random impulse times satisfying the minimum dwell-time condition and we obtain the trajectories depicted in Fig.~\ref{fig:states_minDT} where we can observe the ability of the interval observer to properly frame the trajectory of the system. The disturbance inputs are depicted in Fig.~\ref{fig:inputs_minDT}.

\begin{figure}
  \centering
  \includegraphics[width=0.8\textwidth]{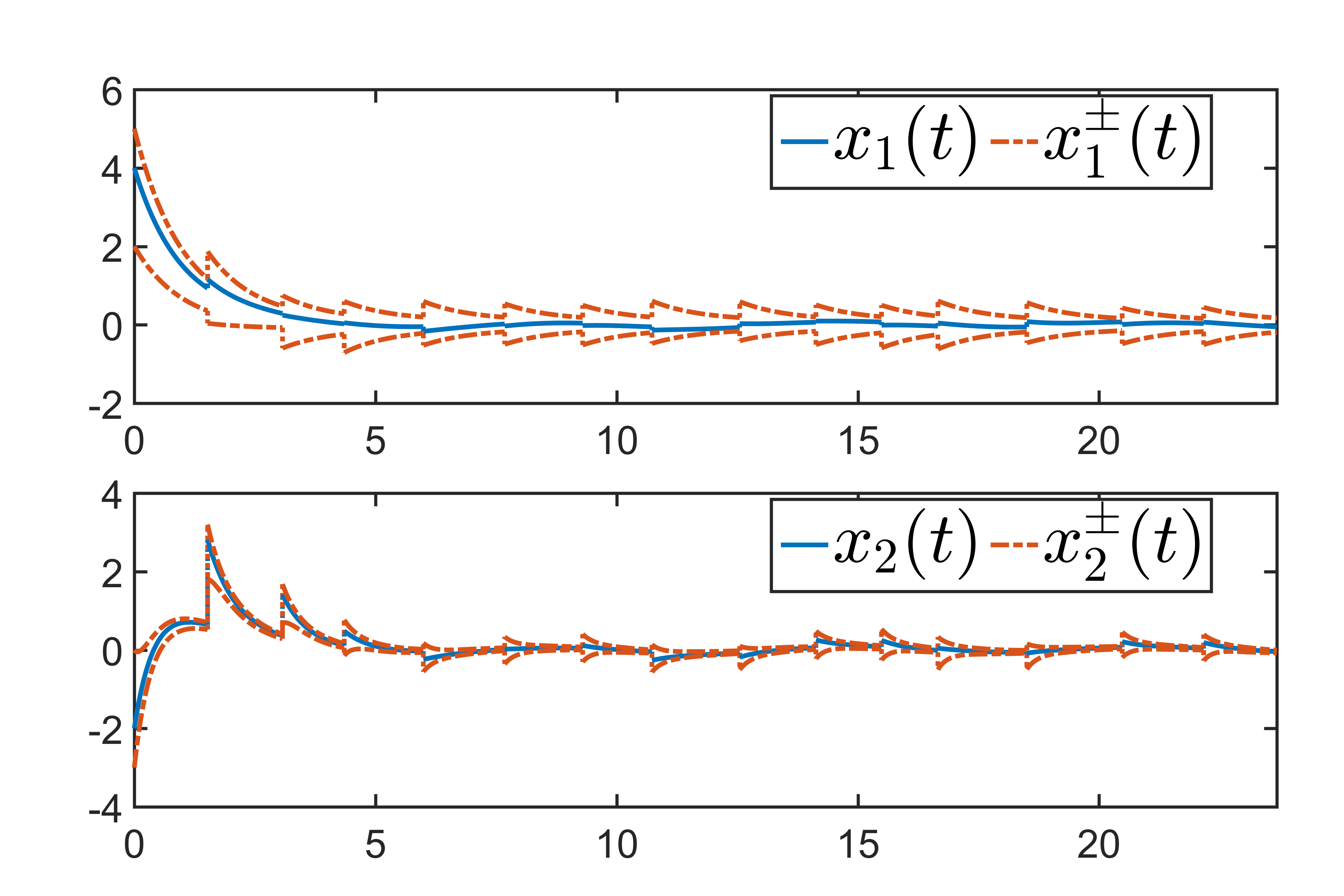}
  \caption{Trajectories of the system \eqref{eq:mainsyst}-\eqref{eq:ex1} and the interval observer \eqref{eq:obs}-\eqref{eq:L2}-\eqref{eq:Lex2} for some randomly chosen impulse times satisfying the minimum dwell-time $\bar{T}=1$.}\label{fig:states_minDT}
\end{figure}

\begin{figure}
  \centering
  \includegraphics[width=0.8\textwidth]{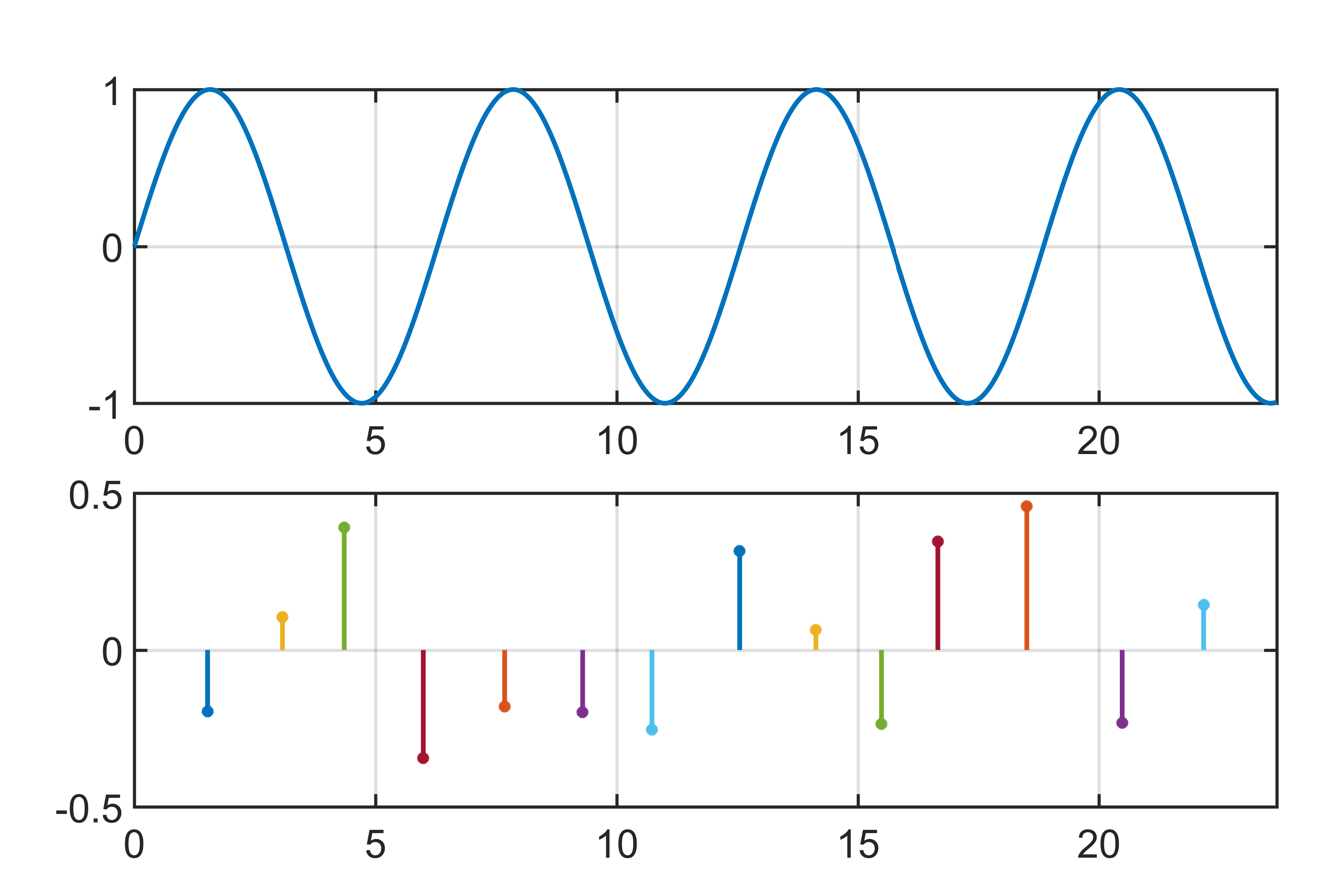}
  \caption{Trajectory of the continuous-time input $w_c$ (top) and the discrete-time input $w_d$ (bottom)}\label{fig:inputs_minDT}
\end{figure}

\subsubsection{Example 3. Minimum dwell-time}

Let us consider here the matrices given by
\begin{equation}\label{eq:ex3}
\begin{array}{l}
    A=\begin{bmatrix}
   -1&     0 &    1  &  -1\\
     1&    -2   &  0&     1\\
     1 &    0&    -3     &1\\
     0   &  1     &2    &-4
  \end{bmatrix}, E_c=\begin{bmatrix}
    0.1\\
    0\\
    -0.1\\
    0
  \end{bmatrix}, J=\begin{bmatrix}
     1&    -1     &0     &0\\
     1  &   3   &  0     &0\\
    -1&0     &1     &0\\
     0    & 1     &0     &1
  \end{bmatrix},  E_d=\begin{bmatrix}
    0.3\\
    0.3\\
    0\\0
  \end{bmatrix},\\
    C_c=\begin{bmatrix}
      0    & 1     &0     &0\\
     0     &0     &0     &1
    \end{bmatrix},C_d=\begin{bmatrix}
    0     &1     &0     &0\\
     1&     0     &0     &0
  \end{bmatrix}, F_c=F_d=\begin{bmatrix}
    0.1\\0
  \end{bmatrix}
\end{array}
\end{equation}
The disturbances are defined in the same way as in the previous example. Choosing a desired minimum dwell-time of $\bar T=1$ and solving the conditions of Theorem \ref{th:2} with polynomials of degree 4 yield the minimum value 3.0324 for $\gamma$ together with the observer gains
\begin{equation}\label{eq:Lex3a}
  L_d=\begin{bmatrix}
   -1&    0.7274\\
    3&    0.2573\\
   0  & -1.4829\\
   0 &  -1.3160
  \end{bmatrix}
\end{equation}
and
\begin{equation}\label{eq:Lex3b}
  \tilde L_c(\tau)=\begin{bmatrix}
 \tilde L_c^1(\tau) & \tilde L_c^2(\tau)
  \end{bmatrix}
\end{equation}
where
\begin{equation}
  \tilde L_c^1(\tau)=\begin{bmatrix}
    1.304\dfrac {- 0.0000001914- 0.1285\tau+ 0.6176{\tau}^{2}- 0.8546{\tau}^{3}+ 0.4655{\tau}^{4}}{ 9.886+ 1.287\tau+ 0.3609{\tau}^{2}- 1.010{\tau}^{3}+ 0.2236{\tau}^{4}}\\
    - 1.304\dfrac {- 0.0000001078- 0.2173\tau+ 1.044{\tau}^{2}- 1.614{\tau}^{3}+ 0.7871{\tau}^{4}}{- 3.148- 0.7717\tau+ 0.1618{\tau}^{2}- 1.682{\tau}^{3}+ 1.374{\tau}^{4}}\\
    \dfrac {3.123+ 0.7654\tau- 1.307{\tau}^{2}+ 3.584{\tau}^{3}- 2.042{\tau}^{4}}{- 3.123- 0.5617\tau+ 0.3280{\tau}^{2}- 2.071{\tau}^{3}+ 1.304{\tau}^{4}}\\
    \dfrac {- 0.0000001113- 0.2332\tau+ 1.120{\tau}^{2}- 1.732{\tau}^{3}+ 0.8442{\tau}^{4}}{ 0.1134- 0.01944\tau+ 0.4850{\tau}^{2}- 1.884{\tau}^{3}+ 2.419{\tau}^{4}}
  \end{bmatrix}
\end{equation}
and
\begin{equation}
  \tilde L_c^2(\tau)=\begin{bmatrix}
    1.304\dfrac {- 7.586- 0.8535\tau- 1.446{\tau}^{2}+ 3.884{\tau}^{3}- 2.712{\tau}^{4}}{ 9.886+ 1.287\tau+ 0.3609{\tau}^{2}- 1.010{\tau}^{3}+ 0.2236{\tau}^{4}}\\
    - 1.304\dfrac { 2.410+ 0.7141\tau- 1.194{\tau}^{2}+ 4.134{\tau}^{3}- 3.378{\tau}^{4}}{- 3.148- 0.7717\tau+ 0.1618{\tau}^{2}- 1.682{\tau}^{3}+ 1.374{\tau}^{4}}\\
    \dfrac { -3.119- 0.6836\tau+ 1.395{\tau}^{2}- 4.907{\tau}^{3}+ 3.622{\tau}^{4}}{- 3.123- 0.5617\tau+ 0.3280{\tau}^{2}- 2.071{\tau}^{3}+ 1.304{\tau}^{4}}\\
    \dfrac { 0.5467+ 0.4849\tau- 2.670{\tau}^{2}+ 4.127{\tau}^{3}+ 1.019{\tau}^{4}}{ 0.1134- 0.01944\tau+ 0.4850{\tau}^{2}- 1.884{\tau}^{3}+ 2.419{\tau}^{4}}
  \end{bmatrix}
\end{equation}
For information, the semidefinite program has 1473 primal variables, 429 dual variables and it takes 13.80 seconds to solve. To illustrate this result, we generate random impulse times satisfying the minimum dwell-time condition and we obtain the trajectories depicted in Fig.~\ref{fig:states_minDT2a} and Fig.~\ref{fig:states_minDT2b} where we can observe the ability of the interval observer to properly frame the trajectory of the system. The disturbance inputs are depicted in Fig.~\ref{fig:inputs_minDT2}.

\begin{figure}
  \centering
  \includegraphics[width=0.8\textwidth]{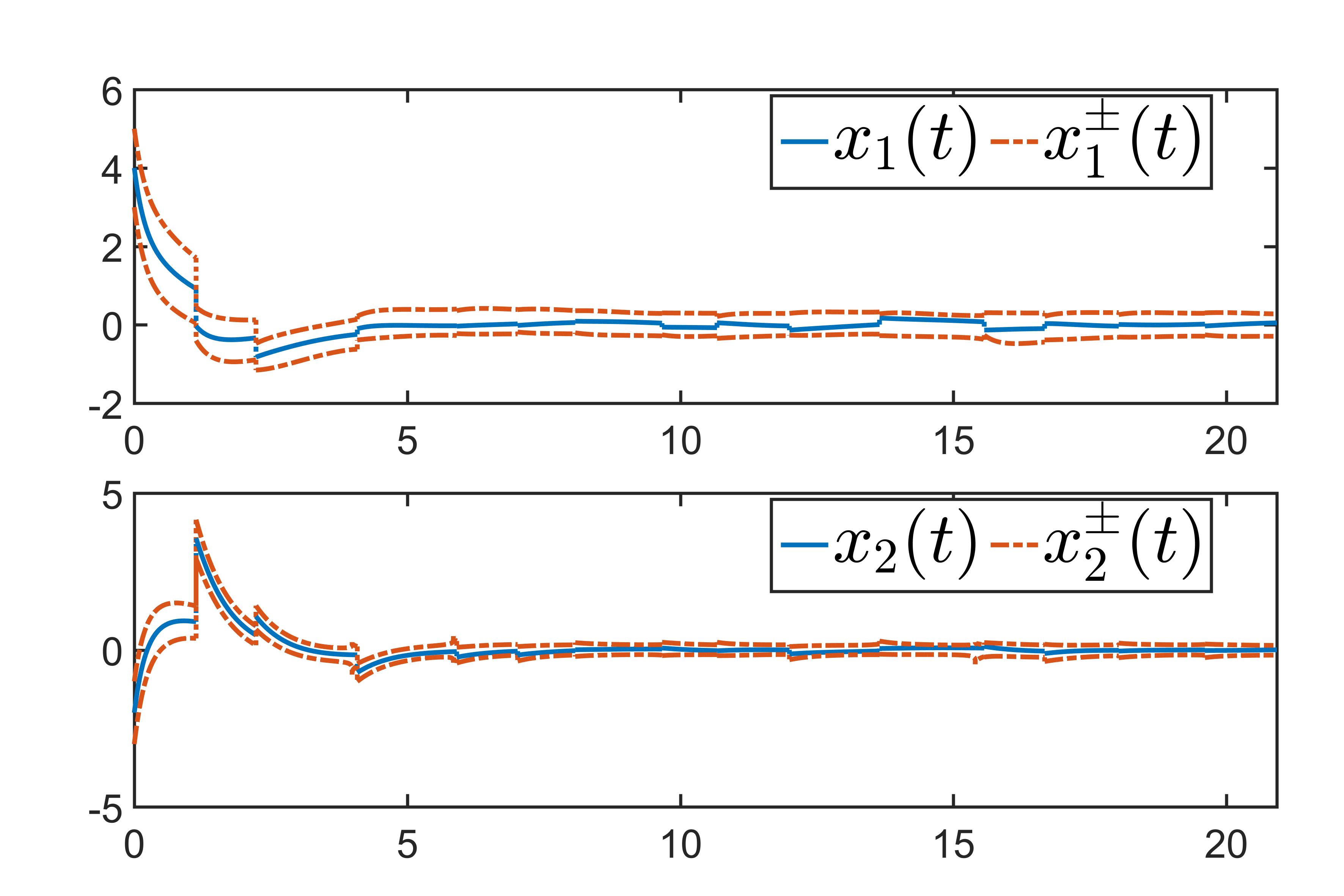}
  \caption{Trajectories of the system \eqref{eq:mainsyst}-\eqref{eq:ex3} and the interval observer \eqref{eq:obs}-\eqref{eq:L2}-\eqref{eq:Lex3a}-\eqref{eq:Lex3b} for some randomly chosen impulse times satisfying the minimum dwell-time $\bar{T}=1$.}\label{fig:states_minDT2a}
\end{figure}

\begin{figure}
  \centering
  \includegraphics[width=0.8\textwidth]{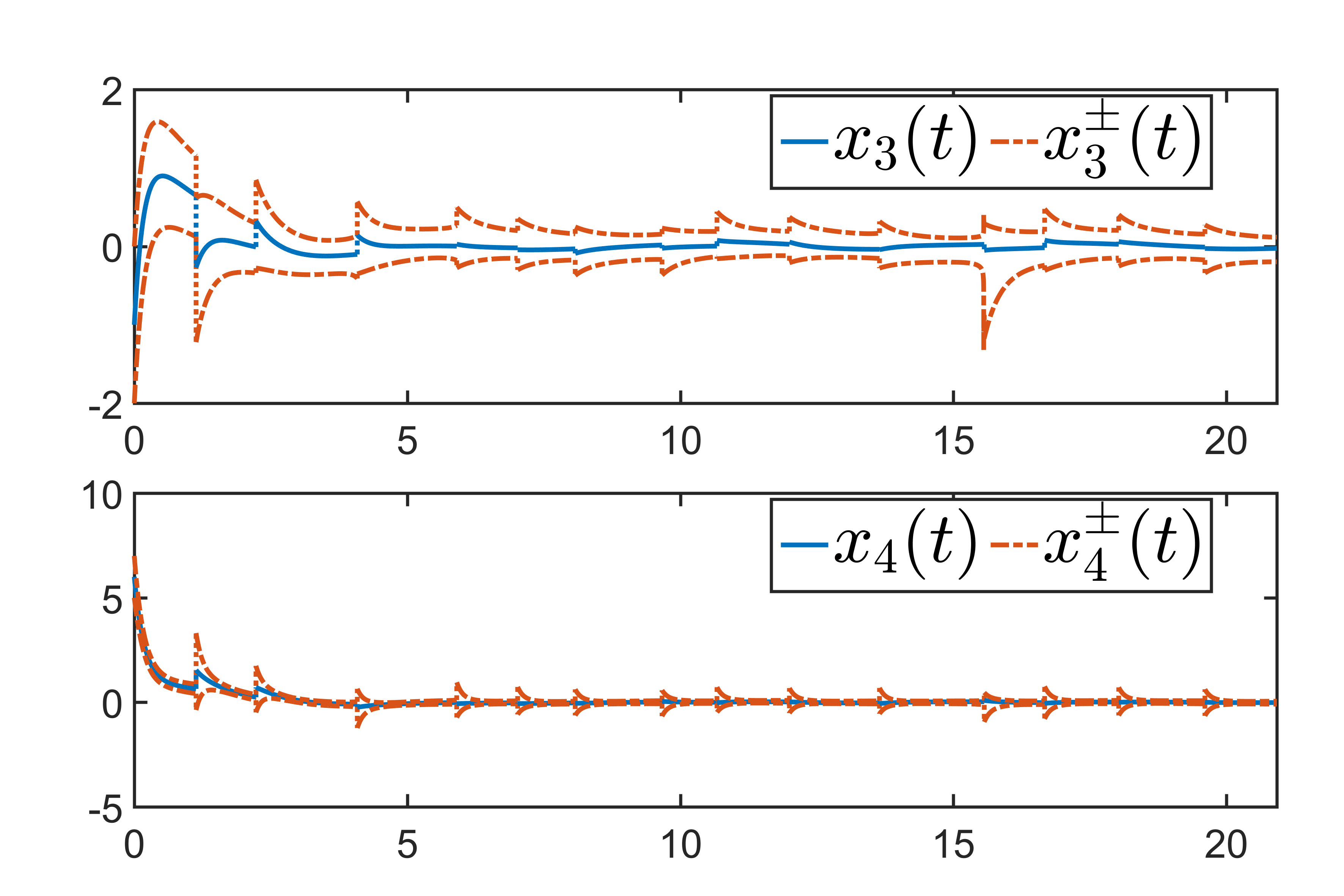}
  \caption{Trajectories of the system \eqref{eq:mainsyst}-\eqref{eq:ex1} and the interval observer \eqref{eq:obs}-\eqref{eq:L2}-\eqref{eq:Lex3a}-\eqref{eq:Lex3b} for some randomly chosen impulse times satisfying the minimum dwell-time $\bar{T}=1$.}\label{fig:states_minDT2b}
\end{figure}

\begin{figure}
  \centering
  \includegraphics[width=0.8\textwidth]{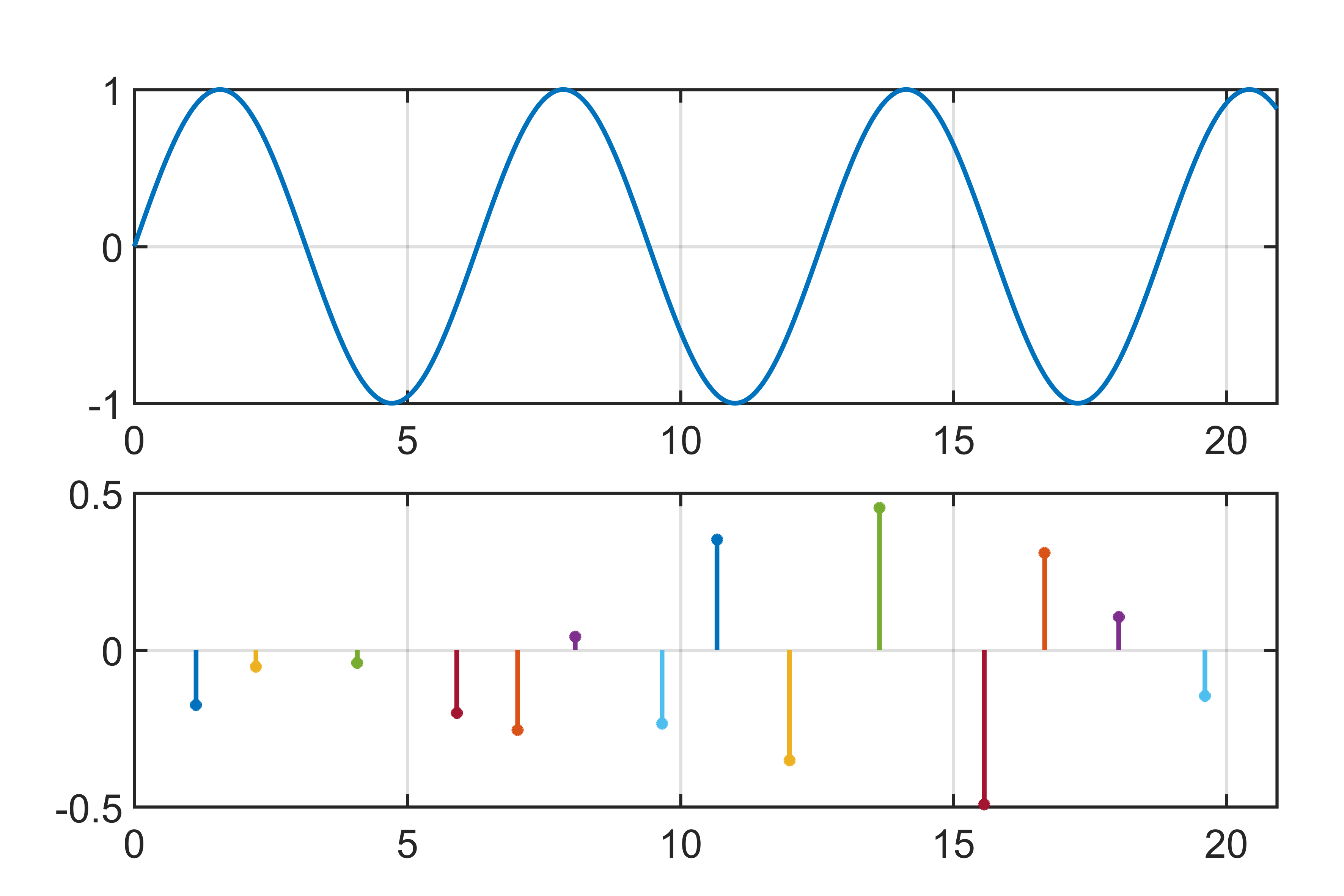}
  \caption{Trajectory of the continuous-time input $w_c$ (top) and the discrete-time input $w_d$ (bottom)}\label{fig:inputs_minDT2}
\end{figure}
}

\section{$L_1$-to-$L_1$ interval observation of linear switched systems}\label{sec:switched}

Let us consider here the switched system
\begin{equation}\label{eq:switched}
\begin{array}{rcl}
    \dot{\tilde{x}}(t)&=&\tilde{A}_{\sigma(t)}\tilde{x}(t)+\tilde{E}_{\sigma(t)}w(t)\\
     \tilde{y}(t)&=&\tilde{C}_{\sigma(t)}\tilde{x}(t)+\tilde{F}_{\sigma(t)}w(t)
\end{array}
\end{equation}
where $\sigma:\mathbb{R}_{\ge0}\mapsto\{1,\ldots,N\}$ is the switching signal, $\tilde x\in\mathbb{R}^n$ is the state of the system, $\tilde{w}\in\mathbb{R}^p$ is the exogenous input and $\tilde{y}\in\mathbb{R}^p$ is the measured output. The switching signal $\sigma$ is assumed to take values in the set $\{1,\ldots,N\}$ and to change values at the times in the sequence $\{t_k\}_{k\ge1}$. This system can be rewritten into the following impulsive system with multiple jump maps as in \cite{Briat:15i,Briat:16c}
\begin{equation}\label{eq:swimp}
\begin{array}{rcl}
    \dot{x}(t)&=&\diag_{i=1}^N(\tilde{A}_{i})x(t)+\col_{i=1}^N(\tilde{E}_{i})w(t)\\
     y(t)&=&\diag_{i=1}^N(\tilde{C}_{i})x(t)+\col_{i=1}^N(\tilde{F}_{i})w(t)\\
     x(t_k^+)&=&J_{ij}x(t_k),\ i,j=1,\ldots,N,\ i\ne j
\end{array}
\end{equation}
where $J_{ij}:=(b_ib_j^T)\otimes I_n$ and $\{b_1,\ldots,b_N\}$ is the standard basis for $\mathbb{R}^N$. It is important to stress that in the above formulation only the part of the state $x(t)$ that evolves according to the subsystem $\sigma(t)$ is meaningful. In this regard, the others can be discarded when plotting the trajectories of the switched system.

Because of the particular structure of the system, we can define w.l.o.g.  an interval observer of the form \eqref{eq:obs} for the system \eqref{eq:swimp} with the gains $L_c(t)=\diag_{i=1}^N(L_c^i(t))$ and $L_d^{ij}=0$. The error dynamics is then given in this case by
\begin{equation}\label{eq:error_switched}
\begin{array}{rcl}
    \dot{e}^\bullet(t)&=&\diag_{i=1}^N(\tilde{A}_{i}-L_c^i(t)\tilde{C}_i)e^\bullet(t)+\col_{i=1}^N(\tilde{E}_{i}-L_c^i(t)\tilde{F}_i)\delta^\bullet(t)\\
   e^\bullet(t_k^+)&=&\left[(b_ib_j^T)\otimes I_n\right]e^\bullet(t_k)\\
   e_c^\bullet(t)&=&\left[I_n\otimes  M\right]e^\bullet(t)
    \end{array}
\end{equation}
where $M\in\mathbb{R}^{n\times n}_{\ge0}$ is a nonzero weighting matrix.

\noindent\blue{Similarly as in the impulsive systems case, we consider the following structure for the gains of the observer:
\begin{equation}\label{eq:L3}
  L_c^i(t)=\left\{\begin{array}{ll}
  \tilde{L}_c^i(t-t_k)& \textnormal{if }t\in(t_k,t_{k}+\tau]\\
  \tilde{L}_c^i(\bar T)& \textnormal{if }t\in(t_k+\bar{T},t_{k+1}]
  \end{array}\right.
\end{equation}
where $\tilde{L}_c^i:\mathbb{R}_{\ge0}\mapsto\mathbb{R}^{n\times q_c}$, $i=1,\dots,N$ are functions to be determined. We can now formally state the considered interval observation problem:}
\begin{problem}\label{problem3}
Find an interval observer of the form \eqref{eq:swimp} (i.e. a matrix-valued function $L_c(\cdot)$ of the form \eqref{eq:L2}  such that the error dynamics \eqref{eq:error_switched} is
  \begin{enumerate}[(a)]
    \item state-positive, that is, for all $i=1,\ldots,N$ we have that
    \begin{itemize}
      \item $\tilde A_i-\tilde L_c^i(\tau) \tilde C_i$ is Metzler for all $\tau\in[0,\bar{T}]$,
      \item $\tilde E_i-\tilde L_c^i(\tau) \tilde F_i $ is nonnegative for all $\tau\in[0,\bar{T}]$,
    \end{itemize}
    \item asymptotically stable under minimum dwell-time $\bar T$ when $w_c\equiv0$, and
    \item satisfies the condition
    \begin{equation}
    ||e_c^\bullet||_{L_1}<\gamma||\delta^\bullet||_{L_1},\ \bullet\in\{-,+\}
  \end{equation}
  for all $w_c\in L_1$ when $x_0=0$ and where $M\in\mathbb{R}^{n\times n}_{\ge0}$.
  \end{enumerate}
\end{problem}

The following result provides a sufficient conditions on whether there exists a solution to Problem \ref{problem3}.
\begin{theorem}\label{th:switched}
There exists a differentiable matrix-valued functions $X_i:[0,\bar T]\mapsto\mathbb{D}^n$, $X_i(\bar T)\succ0$, $i=1,\ldots,N$, matrix-valued functions $U_i:[0,\bar T]\mapsto\mathbb{R}^{n\times q_c}$, $i=1,\ldots,N$, and scalars $\eps,\alpha,\gamma>0$ such that the conditions
\begin{subequations}\label{eq:th2a}
\begin{alignat}{4}
           X_i(\tau)\tilde A_i-U_i(\tau)\tilde C_i+\alpha I_n&\ge0\\
           X_i(\tau)\tilde E_i-U_i(\tau)\tilde F_i&\ge0
  \end{alignat}
\end{subequations}
and
    \begin{subequations}\label{eq:th2b}
\begin{alignat}{4}
           \mathds{1}^T\left[X_i(\bar T)\tilde A_i-U_i(\bar T)\tilde C_i+\eps I_n\right]+\mathds{1}^TM&\le0\\
           \mathds{1}^T\left[\dot{X}_i(\tau)+X_i(\tau)\tilde A_i-U_i(\tau)\tilde C_i\right]+\mathds{1}^TM&\le0\\
            \mathds{1}^T\left[X_j(\bar T)-X_i(0)+\eps I\right]&\le0\\
            \mathds{1}^T\left[X_i(\bar T)\tilde E_i-U_i(\bar T)\tilde F_i\right]-\gamma \mathds{1}^T&\le0\\
            \mathds{1}^T\left[X_i(\tau)\tilde E_i-U_i(\tau)\tilde F_i\right]-\gamma \mathds{1}^T&\le0
\end{alignat}
\end{subequations}
hold for all $\tau\in[0,\bar T]$ and all $i,j=1,\ldots,N$.  Then, there exists an interval observer of the form \eqref{eq:obs}-\eqref{eq:L2} that solves Problem \ref{problem3} and suitable observer gains are given by
\begin{equation}
  \tilde{L}_i(\tau)= X_i(\tau)^{-1}U_i(\tau),\ i=1,\ldots,N.
\end{equation}
\end{theorem}
\begin{proof}
  The proof is performed by direct substitution. See also \cite{Briat:16c}.
\end{proof}

\blue{
\begin{example}
Let us consider the system \eqref{eq:switched} with the matrices.
\begin{equation}\label{eq:ex3}
\begin{array}{l}
    \tilde A_1=\begin{bmatrix}
    -1 & 0\\
    -1 & -2
  \end{bmatrix}, \tilde E_1=\begin{bmatrix}
    0.1\\
    -0.1
  \end{bmatrix}, \tilde A_2=\begin{bmatrix}
-1 & -1\\
1 & -6
  \end{bmatrix}, \tilde E_2=\begin{bmatrix}
    0.5\\
    0
  \end{bmatrix},
    \\ \tilde C_1=\begin{bmatrix}
    1 & 0
  \end{bmatrix},\tilde C_2=\begin{bmatrix}
    0 & 1
  \end{bmatrix}, \tilde F_c=\tilde F_d=0.1.
\end{array}
\end{equation}
Solving for the conditions in Theorem \ref{th:switched} with polynomials of degree 2 and a minimum dwell-time equal to $\bar T=1$, we get the minimum $\gamma=0.80023$. The number of primal/dual variables is 268/93 and the problem is solved in 2.86 seconds. The observer gains are not given for brevity. The trajectories of the system and the interval observer are depicted in Fig.~\ref{fig:states_minDT_switched}. The disturbance input and the switching signal are depicted in Fig.~\ref{fig:inputs_minDT_switched}.

\begin{figure}
  \centering
  \includegraphics[width=0.8\textwidth]{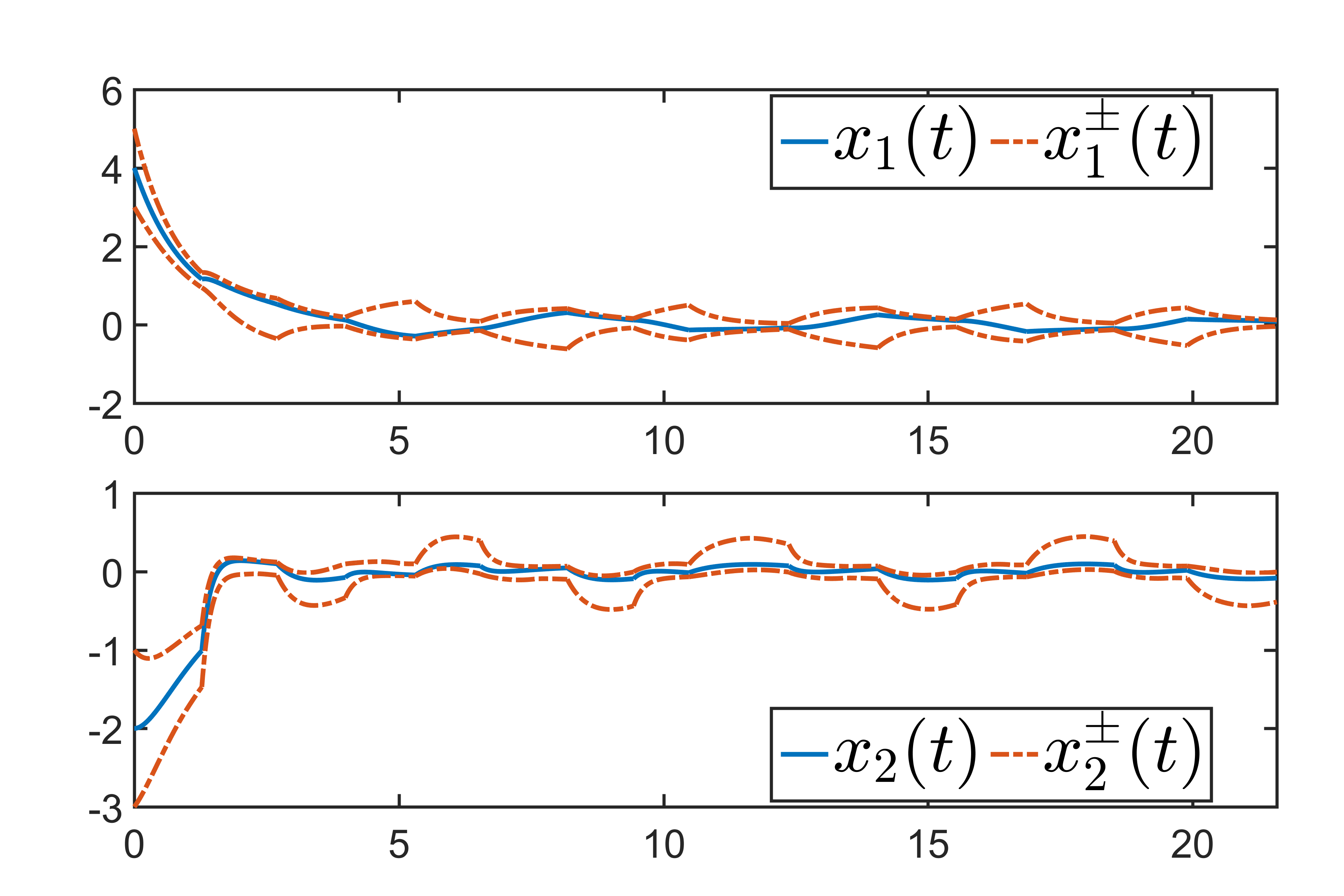}
  \caption{Trajectories of the system \eqref{eq:switched}-\eqref{eq:ex3} and the interval observer \eqref{eq:swimp} for some randomly chosen impulse times satisfying the minimum dwell-time $\bar T=1$.}\label{fig:states_minDT_switched}
\end{figure}

\begin{figure}
  \centering
  \includegraphics[width=0.8\textwidth]{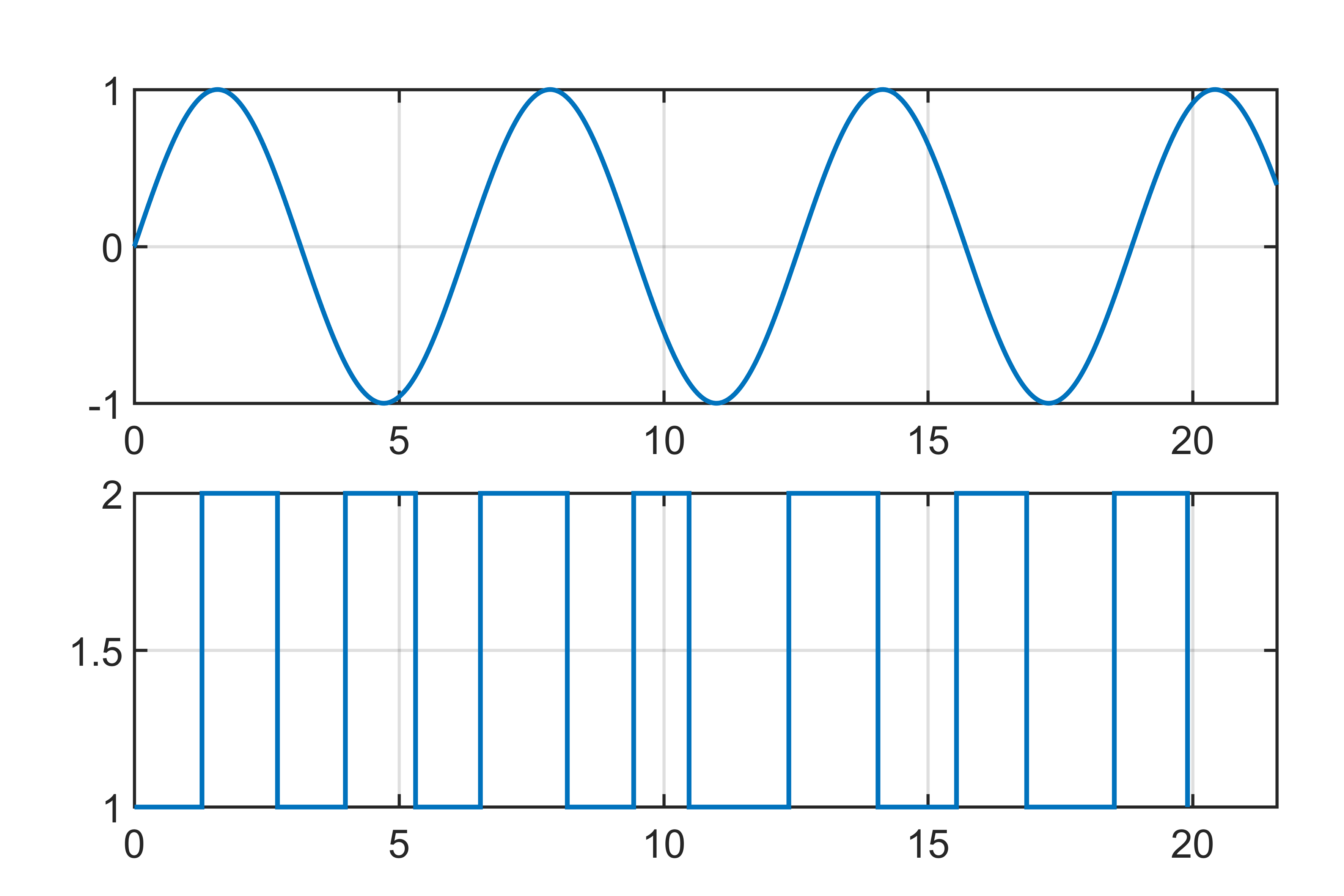}
  \caption{Trajectory of the continuous-time input $w_c$ (top) and the switching signal $\sigma$ (bottom)}\label{fig:inputs_minDT_switched}
\end{figure}
\end{example}}

\section{Future works}

An interesting problem is whether the design of $L_\infty/\ell_\infty$-to-$L_\infty/\ell_\infty$ interval observers for linear impulsive and switched systems can be similarly formulated. The first difficulty is that in \cite{Briat:11h,Briat:15g}, the $L_\infty$-gain is computed as the $L_1$-gain of the transposed system. It is unclear whether this is also the case for linear positive impulsive systems and is worthy of investigation. The next problem will be obtaining linear conditions as it is also unclear whether the procedure considered in \cite{Briat:15g} will still be applicable.

\blue{Potential improvements to this approach lie at the level of its conservatism reduction. The conservatism can be reduced by (i) considering more complex observer structures and (ii) the use of changes of variables, as it is customary in the literature on interval-observers, and (iii) the consideration of higher-order polynomial Lyapunov functions, possibly of odd degree. The positive point is that the obtained conditions can be readily adapted to cope with the extensions (i) and (ii). The derivation design conditions using higher order polynomial Lyapunov functions is more complex. Indeed, the benefits of using linear ones lies in the fact that the hybrid $L_1/\ell_1$-gain can be easily characterized, that structural constraints on the gains of the observer can be easily incorporated, and that the computational complexity is kept low. This is not the case for higher-order Lyapunov functions, even quadratic ones. This latter point is left for future research.}

%
%
%



\begin{thebibliography}{10}
\providecommand{\url}[1]{#1}
\csname url@samestyle\endcsname
\providecommand{\newblock}{\relax}
\providecommand{\bibinfo}[2]{#2}
\providecommand{\BIBentrySTDinterwordspacing}{\spaceskip=0pt\relax}
\providecommand{\BIBentryALTinterwordstretchfactor}{4}
\providecommand{\BIBentryALTinterwordspacing}{\spaceskip=\fontdimen2\font plus
\BIBentryALTinterwordstretchfactor\fontdimen3\font minus
  \fontdimen4\font\relax}
\providecommand{\BIBforeignlanguage}[2]{{%
\expandafter\ifx\csname l@#1\endcsname\relax
\typeout{** WARNING: IEEEtran.bst: No hyphenation pattern has been}%
\typeout{** loaded for the language `#1'. Using the pattern for}%
\typeout{** the default language instead.}%
\else
\language=\csname l@#1\endcsname
\fi
#2}}
\providecommand{\BIBdecl}{\relax}
\BIBdecl

\bibitem{Gouze:00}
J.~L. Gouz{\'{e}}, A.~Rapaport, and M.~Z. {Hadj-Sadok}, ``Interval observers
  for uncertain biological systems,'' \emph{Ecological modelling}, vol. 133,
  pp. 45--56, 2000.

\bibitem{Mazenc:11}
F.~Mazenc and O.~Bernard, ``Interval observers for linear time-invariant
  systems with disturbances,'' \emph{Automatica}, vol.~47, pp. 140--147, 2011.

\bibitem{Briat:15g}
C.~Briat and M.~Khammash, ``Interval peak-to-peak observers for continuous- and
  discrete-time systems with persistent inputs and delays,'' \emph{Automatica},
  vol.~74, pp. 206--213, 2016.

\bibitem{Mazenc:12}
F.~Mazenc, M.~Kieffer, and E.~Walter, ``Interval observers for continuous-time
  linear systems,'' in \emph{American Control Conference}, Montr{\'{e}}al,
  Canada, 2012, pp. 1--6.

\bibitem{Combastel:13}
C.~Combastel, ``Stable interval observers in {${\mathbb{C}}$} for linear
  systems with time-varying input bounds,'' \emph{IEEE Transactions on
  Automatic Control}, vol. 58(2), pp. 481--487, 2013.

\bibitem{Cacace:15}
F.~Cacace, A.~Germani, and C.~Manes, ``A new approach to design interval
  observers for linear systems,'' \emph{IEEE Transactions on Automatic
  control}, vol. 60(6), pp. 1665--1670, 2015.

\bibitem{Thabet:14}
R.~E.~H. Thabet, T.~Ra{\"{i}}ssi, C.~Combastel, D.~Efimov, and A.~Zolghadri,
  ``An effective method to interval observer design for time-varying systems,''
  \emph{Automatica}, vol. 50(10), pp. 2677--2684, 2014.

\bibitem{Efimov:13c}
D.~Efimov, W.~Perruquetti, and J.-P. Richard, ``On reduced-order interval
  observers for time-delay systems,'' in \emph{12th European Control
  Conference}, Z{\"{u}}rich, Switzerland, 2013, pp. 2116--2121.

\bibitem{Efimov:13b}
D.~Efimov, T.~Ra{\"{i}}ssi, and A.~Zolghadri, ``Control of nonlinear and {LPV}
  systems: Interval observer-based framework,'' \emph{IEEE Transactions on
  Automatic Control}, vol. 58(3), pp. 773--778, 2013.

\bibitem{Chebotarev:15}
S.~Chebotarev, D.~Efimov, T.~Ra{\"{i}}ssi, and A.~Zolghadri, ``Interval
  observers for continuous-time {LPV} systems with ${L}_1$/${L}_2$
  performance,'' \emph{Automatica}, vol.~58, pp. 82--89, 2015.

\bibitem{Mazenc:13}
F.~Mazenc, T.~N. Dinh, and S.-I. Niculescu, ``Robust interval observers and
  stabilization design for discrete-time systems with input and output,''
  \emph{Automatica}, vol.~49, pp. 3490--3497, 2013.

\bibitem{Mazenc:14b}
F.~Mazenc and T.~N. Dinh, ``Construction of interval observers for
  continuous-time systems with discrete measurements,'' \emph{Automatica},
  vol.~50, pp. 2555--2560, 2014.

\bibitem{Efimov:16}
D.~Efimov, E.~Fridman, A.~Polyakov, W.~Perruquetti, and J.-P. Richard, ``On
  design of interval observers with sampled measurement,'' \emph{Systems \&
  Control Letters}, vol.~96, pp. 158--164, 2016.

\bibitem{Sivashankar:94}
N.~Sivashankar and P.~P. Khargonekar, ``Characterization of the
  {${\mathcal{L}}_2$}-induced norm for linear systems with jumps with
  applications to sampled-data systems,'' \emph{SIAM Journal on Control and
  Optimization}, vol. 32(4), pp. 1128--1150, 1994.

\bibitem{Naghshtabrizi:08}
P.~Naghshtabrizi, J.~P. Hespanha, and A.~R. Teel, ``Exponential stability of
  impulsive systems with application to uncertain sampled-data systems,''
  \emph{Systems \& Control Letters}, vol.~57, pp. 378--385, 2008.

\bibitem{Briat:11l}
C.~Briat and A.~Seuret, ``A looped-functional approach for robust stability
  analysis of linear impulsive systems,'' \emph{Systems \& Control Letters},
  vol. 61(10), pp. 980--988, 2012.

\bibitem{Goebel:12}
R.~Goebel, R.~G. Sanfelice, and A.~R. Teel, \emph{Hybrid Dynamical Systems.
  Modeling, Stability, and Robustness}.\hskip 1em plus 0.5em minus 0.4em\relax
  Princeton University Press, 2012.

\bibitem{Briat:15i}
C.~Briat, ``Stability analysis and stabilization of stochastic linear
  impulsive, switched and sampled-data systems under dwell-time constraints,''
  \emph{Automatica}, vol.~74, pp. 279--287, 2016.

\bibitem{Geromel:15}
J.~C. Geromel and M.~Souza, ``On an {LMI} approach to optimal sampled-data
  state feedback control design,'' \emph{International Journal of Control},
  vol. 88(11), pp. 2369--2379, 2015.

\bibitem{Briat:16c}
C.~Briat, ``Dwell-time stability and stabilization conditions for linear
  positive impulsive and switched systems,'' \emph{Nonlinear Analysis: Hybrid
  Systems}, vol.~24, pp. 198--226, 2017.

\bibitem{Degue:16nolcos}
K.~H. Degue, D.~Efimov, and J.-P. Richard, ``Interval observers for linear
  impulsive systems,'' in \emph{10th IFAC Symposium on Nonlinear Control
  Systems}, 2016.

\bibitem{Briat:17ifacObs}
C.~Briat and M.~Khammash, ``Simple interval observers for linear impulsive
  systems with applications to sampled-data and switched systems,'' in
  \emph{20th IFAC World Congress}, Toulouse, France, 2017, pp. 5235--5240.

\bibitem{Rabehi:17}
D.~Rabehi, D.~Efimov, and J.-P. Richard, ``Interval estimation for linear
  switched systems,'' in \emph{20th IFAC World Congress}, Toulouse, France,
  2017, pp. 6265--6270.

\bibitem{Ethabet:17}
H.~Ethabet, T.~Raissi, M.~Amairi, and M.~Aoun, ``Interval observers design for
  continuous-time linear switched systems,'' in \emph{20th IFAC World
  Congress}, Toulouse, France, 2017, pp. 6259--6264.

\bibitem{Briat:11h}
C.~Briat, ``Robust stability and stabilization of uncertain linear positive
  systems via integral linear constraints - ${L_1}$- and ${L_\infty}$-gains
  characterizations,'' \emph{{I}nternational {J}ournal of {R}obust and
  {N}onlinear {C}ontrol}, vol. 23(17), pp. 1932--1954, 2013.

\bibitem{Briat:13d}
------, ``Convex conditions for robust stability analysis and stabilization of
  linear aperiodic impulsive and sampled-data systems under dwell-time
  constraints,'' \emph{Automatica}, vol. 49(11), pp. 3449--3457, 2013.

\bibitem{Briat:14f}
------, ``Convex conditions for robust stabilization of uncertain switched
  systems with guaranteed minimum and mode-dependent dwell-time,''
  \emph{Systems \& Control Letters}, vol.~78, pp. 63--72, 2015.

\bibitem{Briat:15f}
------, ``Theoretical and numerical comparisons of looped functionals and
  clock-dependent {L}yapunov functions - {T}he case of periodic and
  pseudo-periodic systems with impulses,'' \emph{International Journal of
  Robust and Nonlinear Control}, vol.~26, pp. 2232--2255, 2016.

\bibitem{Allerhand:11}
L.~I. Allerhand and U.~Shaked, ``Robust stability and stabilization of linear
  switched systems with dwell time,'' \emph{{IEEE} Transactions on Automatic
  Control}, vol. 56(2), pp. 381--386, 2011.

\bibitem{Handelman:88}
D.~Handelman, ``Representing polynomials by positive linear functions on
  compact convex polyhedra,'' \emph{Pacific Journal of Mathematics}, vol.
  132(1), pp. 35--62, 1988.

\bibitem{Putinar:93}
M.~Putinar, ``Positive polynomials on compact semi-algebraic sets,''
  \emph{Indiana Univ. Math. J.}, vol.~42, no.~3, pp. 969--984, 1993.

\bibitem{Parrilo:00}
P.~Parrilo, ``Structured semidefinite programs and semialgebraic geometry
  methods in robustness and optimization,'' Ph.D. dissertation, California
  Institute of Technology, Pasadena, California, 2000.

\bibitem{sostools3}
A.~Papachristodoulou, J.~Anderson, G.~Valmorbida, S.~Prajna, P.~Seiler, and
  P.~A. Parrilo, \emph{{SOSTOOLS}: Sum of squares optimization toolbox for
  {MATLAB} v3.00}, 2013.

\bibitem{Briat:11g}
C.~Briat, ``Robust stability analysis of uncertain linear positive systems via
  integral linear constraints - ${L_1}$- and ${L_\infty}$-gains
  characterizations,'' in \emph{50th {IEEE} Conference on Decision and
  Control}, Orlando, Florida, USA, 2011, pp. 3122--3129.

\bibitem{Ebihara:11}
Y.~Ebihara, D.~Peaucelle, and D.~Arzelier, ``${L_1}$ gain analysis of linear
  positive systems and its applications,'' in \emph{50th Conference on Decision
  and Control, Orlando, Florida, USA}, 2011, pp. 4029--4034.

\bibitem{Xiang:15a}
W.~Xiang, ``On equivalence of two stability criteria for continuous-time
  switched systems with dwell time constraint,'' \emph{Automatica}, vol.~54,
  pp. 36--40, 2015.

\bibitem{Raissi:12}
T.~Ra{\"{i}}ssi, D.~Efimov, and A.~Zolghadri, ``Interval state estimation for a
  class of nonlinear systems,'' \emph{IEEE Transactions on Automatic Control},
  vol.~57, pp. 260--265, 2012.

\bibitem{Chambon:15}
E.~Chambon, P.~Apkarian, and L.~Burlion, ``Metzler matrix transform
  determination using a non-smooth optimization technique with an application
  to interval observers,'' in \emph{SIAM Conference on Control and its
  Applications}, Paris, France, 2015.

\bibitem{Raissi:18}
T.~Ra{\:i}ssi and D.~Efimov, ``Some recent results on the design and
  implementation of interval observers for uncertain systems,'' \emph{at -
  Automatisierungstechnik}, vol. 66(3), pp. 213--224, 2018.

\bibitem{Chesi:10b}
G.~Chesi, ``{LMI techniques for optimization over polynomials in control: A
  survey},'' \emph{IEEE Transactions on Automatic Control}, vol. 55(11), pp.
  2500--2510, 2010.

\bibitem{Sturm:01a}
J.~F. Sturm, ``Using {SEDUMI} $1. 02$, a {M}atlab {T}oolbox for {O}ptimization
  {O}ver {S}ymmetric {C}ones,'' \emph{Optimization Methods and Software},
  vol.~11, no.~12, pp. 625--653, 2001.

\end{thebibliography}

\end{document}